\documentclass[11pt,a4paper,reqno]{amsart}
\usepackage{amsmath, amssymb, latexsym, enumerate,  graphicx,tikz, stmaryrd, eufrak,enumitem}

\usepackage{thmtools}

\usepackage{hyperref}
\usepackage{cleveref}
\usepackage{bbm}
\usepackage{cases}
\usepackage{array}
\usepackage{multirow}
\usepackage{tikz-cd}
\usepackage[all]{xy}

\usepackage{mathtools}

\usepackage[lighttt]{lmodern}

\usepackage{listings}
\usepackage[hmarginratio={1:1},vmarginratio={1:1},lmargin=80.0pt,tmargin=90.0pt]{geometry}

\usepackage{dynkin-diagrams}

\usepackage{tikz}
\tikzset{anchorbase/.style={baseline={([yshift=-0.5ex]current bounding box.center)}}}
\usetikzlibrary{decorations.markings}
\usetikzlibrary{decorations.pathreplacing}
\usetikzlibrary{arrows,shapes,positioning,backgrounds}
\tikzstyle directed=[postaction={decorate,decoration={markings,
    mark=at position #1 with {\arrow{>}}}}]
\tikzstyle rdirected=[postaction={decorate,decoration={markings,
    mark=at position #1 with {\arrow{<}}}}]
    
\usepackage{graphicx}
\usepackage{nicefrac}

 \newlength{\baseunit}               
 \newcount{\numlines}                
\newtheorem{thm}{Theorem}
\newtheorem{ques}[thm]{Question}

\newtheorem{theorem}[subsection]{Theorem}
\newtheorem{lemma}[theorem]{Lemma}
\newtheorem{prop}[theorem]{Proposition}
\newtheorem{corollary}[subsection]{Corollary}

\theoremstyle{definition}

\newtheorem{remark}[theorem]{Remark}

\newcommand{\mG}{\mathbb{G}}

\newcommand{\bk}{\mathbf{k}}

\newcommand{\cC}{\mathcal{C}}
\newcommand{\cD}{\mathcal{D}}

\numberwithin{equation}{section}

\newcommand{\ev}{\mathrm{ev}}
\newcommand{\coev}{\mathrm{coev}}
\newcommand{\braid}{\mathrm{braid}}

\newcommand{\im}{\mathrm{Im}}

\newcommand{\charr}{\mathrm{char}}

\newcommand{\Ver}{\mathsf{Ver}}

\newcommand{\Mod}{\mathsf{Mod}}

\newcommand{\Rep}{\mathsf{Rep}}

\newcommand{\unit}{{\mathbbm{1}}}

\newcommand{\cO}{\mathcal{O}}

\newcommand{\mN}{\mathbb{N}}
\newcommand{\mZ}{\mathbb{Z}}

\newcommand{\mC}{\mathbb{C}}

\newcommand{\mQ}{\mathbb{Q}}
\newcommand{\fg}{\mathfrak{g}}

\newcommand{\fsl}{\mathfrak{sl}}
\newcommand{\fgl}{\mathfrak{gl}}
\newcommand{\fso}{\mathfrak{so}}
\newcommand{\fsp}{\mathfrak{sp}}
\newcommand{\fosp}{\mathfrak{osp}}

\newcommand{\End}{\mathrm{End}}

\newcommand{\Hom}{\mathrm{Hom}}
\newcommand{\Lie}{\mathrm{Lie}}

\newcommand{\Sym}{\mathrm{Sym}}
\newcommand{\id}{\mathrm{id}}

\newcommand{\Vecc}{\mathsf{Vec}}

\newcommand{\sVec}{\mathsf{sVec}}

\newcommand{\Indd}{\mathsf{Ind}}

\newcommand{\Tilt}{\mathsf{Tilt}}

\newcommand{\mF}{\mathbb{F}}

\newcommand{\Id}{\mathrm{Id}}
\newcommand{\Aut}{\mathrm{Aut}}
\newcommand{\ad}{\mathrm{ad}}

\newcommand{\GL}{\mathrm{GL}}

\newcommand{\SL}{\mathrm{SL}}

\newcommand{\SO}{\mathrm{SO}}

\newcommand{\Sp}{\mathrm{Sp}}

\newcommand{\incl}{\mathrm{incl}}
\newcommand{\proj}{\mathrm{proj}}

\newcommand{\sM}[1]{\left(\begin{smallmatrix}#1\end{smallmatrix}\right)}

\begin{document}
\title{Small Lie algebras in the Verlinde category}
\author{Joseph Newton}
\address{School of Mathematics and Statistics, University of Sydney, Australia}
\email{j.newton@maths.usyd.edu.au}
\date{\today}


\setcounter{tocdepth}{1}

\begin{abstract}
This paper studies simple Lie algebras in the Verlinde category, with canonical fundamental group action, using combinatorial methods. We give an exhaustive list of such Lie algebras which either have length at most 3, or appear as a subalgebra of the restriction of the Lie algebra of a simple algebraic group along a principal morphism. The results and exceptions answer several questions and conjectures regarding Verlinde categories of algebraic groups and their Lie algebras.
\end{abstract}

\maketitle

\tableofcontents

\section*{Introduction}

Let $\bk$ be an algebraically closed field of characteristic $p$, and let $\Ver_p$ be the semisimple Verlinde category, which is the semisimplification of either the category of tilting modules of $\SL_2$ or the category of finite-dimensional representations of $\mZ/p\mZ$. By the theory of Tannaka-Krein duality from \cite{De} (see \cite[\S4]{incompressible} for a more complete treatment), any symmetric tensor category which fibres over some symmetric tensor category $\cD$ can be understood as representations of an affine group scheme in $\cD$, and the case $\cD=\Ver_p$ is especially interesting for a number of reasons. Firstly, $\Ver_p$ is incompressible, meaning it does not fibre over any smaller symmetric tensor category. Secondly, all Frobenius exact (e.g. semisimple) symmetric tensor categories of moderate growth fibre over $\Ver_p$ by \cite{frobenius}. And thirdly, it is shown in \cite{Ve1} that affine group schemes in $\Ver_p$ admit a theory of Harish-Chandra pairs, similarly to supergroups, so a lot of the data of such a group is contained in its Lie algebra. Thus it is extremely relevant to the wider study of tensor categories to understand Lie algebras in $\Ver_p$. Recent results studying the structure and classification of Lie algebras in $\Ver_p$ (and its generalisation $\Ver_{p^n}$) can be found in \cite{CEN,EKO,Hu}, and results regarding Lie algebras in more general symmetric tensor categories include \cite{APS,BP}.

Tannaka-Krein duality introduces an additional piece of data to an affine group scheme $G$ over a category $\cD$, which is a morphism $\pi(\cD)\to G$ where $\pi(\cD)$ is the fundamental group of $\cD$, inducing the canonical action of $\pi(\cD)$ on $\cO(G)$. Thus, we are interested in studying Lie algebras $\fg$ in $\cD$ equipped with a morphism $\Lie(\pi(\cD))\to\fg$ inducing the canonical action, which we call a ``based Lie algebra.'' For $\cD=\Ver_p$ we will typically omit the morphism, since we show in Proposition~\ref{PropUnique} that it is unique. Since the representation theory of arbitrary Lie algebras in $\Ver_p$ is encompassed by the representation theory of based Lie algebras (see Remark~\ref{RemBased}), we focus on the based case in this paper.

The main result of this paper is the following theorem. This theorem proves \cite[Conjecture~4.4]{asymptotic}, a conjectural classification of length 2 simple based Lie algebras in $\Ver_p$, and also answers the analogous question about length 3. It also gives a complete answer to \cite[Question~6.5]{EO} regarding surjective tensor functors out of the categories $\Ver_p(G)$ for simple algebraic groups $G$. Since $\Ver_2$ and $\Ver_3$ are equivalent to vector spaces and super vector spaces respectively, we focus on $p\geq5$. For a simple algebraic group $G$ and a map $\phi:\SL_2\to G$ such that $\Lie(G)$ is a tilting module of $\SL_2$ under the action coming from $\phi$, we write $\Lie_p(G,\phi)$ for the image of $\Lie(G)$ in $\Ver_p$.

\begin{thm}\label{ThmMain}
Let $\bk$ be an algebraically closed field with $\charr\bk=p\geq5$. If $\fg$ is a based Lie algebra in $\Ver_p$ satisfying either of the following conditions:
\begin{enumerate}[label=(\alph*)]
\item $\fg$ is simple, has length at most 3, and is not in $\sVec\subset\Ver_p$; or
\item $\fg$ is a subalgebra of $\Lie_p(G',\phi')$ for some simple algebraic group $G'$ and a principal morphism $\phi':\SL_2\to G'$,
\end{enumerate}
then $\fg\cong\Lie_p(G,\phi)$ for some simple algebraic group $G$ and morphism $\phi:\SL_2\to G$.
\end{thm}

All Lie algebras satisfying (a) or (b) are explicitly enumerated in Theorems~\ref{ThmSubalgs}, \ref{ThmLength1}, \ref{ThmLength2} and \ref{ThmLength3}. In every case except for one, the morphism $\phi$ is principal (meaning its image contains a regular unipotent element of $G$), and $\Rep_{\Ver_p^+}(\fg,\Lie(\phi))\simeq\Ver_p(G)$ for adjoint-type $G$, by the results of \cite[\S3]{CEN}. The exception is $\Lie_{29}(E_8,a_1)$ in $\Ver_{29}^+$ with underlying object $L_2\oplus L_{14}\oplus L_{26}$, where $a_1$ corresponds to a certain subregular nilpotent in $\Lie(G)$ (see Remark~\ref{RemE8a1}). We observe that this appears as a subalgebra of $\Lie_{29}(E_7)$, which allows us to show that its representation category is semisimple and hence a counterexample to \cite[Conjecture~4.1]{asymptotic}.

In Section~\ref{SecComp} we report the results of a computational search for simple based Lie algebras in $\Ver_p$ on objects of length 4 for $p\leq67$. We find that many of these arise from non-principal restrictions of ordinary algebraic groups in a similar manner to $\Lie_{29}(E_8,a_1)$. There are also two types of length 4 simple based Lie algebras which are not the restriction of a simple algebraic group at all, and can instead be obtained from the queer Lie superalgebra $Q(2)$ and the Hamiltonian Lie superalgebra $H(5)$. This gives a negative answer to \cite[Question~4.6]{asymptotic}, and suggests a replacement question:

\begin{ques}
Is every simple based Lie algebra in $\Ver_p$ isomorphic to the restriction of a simple Lie superalgebra $\fg$ to $\Ver_p$ via some morphism $\fsl_2\to\fg$?
\end{ques}

The methodology in this paper is a refinement of \cite[\S4.2]{CEN}. Specifically, we give a complete combinatorial description of the monoidal structure of $\Ver_p$ with respect to fixed direct sum decompositions of tensor products of simple objects, and this approach may have further applications.

The layout of the paper is as follows. Section~\ref{SecPrelim} establishes combinatorial formulas for certain $\SL_2$-module morphisms which will be used throughout. Section~\ref{SecLieAlgs} defines based Lie algebras, and characterises them in $\Ver_p$. Section~\ref{SecImages} discusses images of ordinary Lie algebras in $\Ver_p$, proves part (b) of Theorem~\ref{ThmMain}, and discusses the exceptional $E_8$, $Q(2)$ and $H(5)$ cases. Section~\ref{SecLieForm} sets up more general formulas for Lie algebras in $\Ver_p$, and starts the proof of part (a) of Theorem~\ref{ThmMain}. Section~\ref{SecLength3} completes the proof, focussing on length 3 Lie algebras. Finally, Section~\ref{SecComp} states some computational results for length 4 Lie algebras.

\subsection*{Acknowledgements} This work is part of PhD research supervised by Kevin Coulembier, and was supported by an Australian Government Research Training Program Scholarship. The author would also like to thank Pavel Etingof for helpful correspondence, especially in relation to Remark~\ref{RemE8a1} and Proposition~\ref{PropHQ}.

The large language model GPT-5.6 Sol was used for writing computer programs, finding references, and proofreading. All material in this paper was written by the author. 

\section{$\SL_2$ combinatorics}\label{SecPrelim}

Let us first establish combinatorial formulas characterising the monoidal structure in $\Rep\SL_2$. These will be used to describe the structure of $\Ver_p$ explicitly.

\subsection{Tensor products}\label{SecInclProj} Let $\bk$ be a field of any characteristic. For the algebraic group $\SL_2$ over $\bk$, we fix the Borel subgroup of upper triangular matrices and torus of diagonal matrices, and identify the weight space with $\mZ$. We write $\Rep\SL_2$ for the category of finite-dimensional representations of $\SL_2$. If $x$ and $y$ are the basis vectors for the standard 2-dimensional representation $V_1$ of $\SL_2$, then we can identify $\Sym V_1$ with the polynomial space $\bk[x,y]$. Then the dual Weyl module $V_n$ with highest weight $n\in\mN$ is the subspace of $\bk[x,y]$ spanned by monomials of degree $n$. We write $v_{n,i}$ (or just $v_i$) for the basis vector $x^iy^{n-i}$ in $V_n$.

We say that a triple $(a,b,c)\in\mN^3$ satisfies the \textbf{$0$-triangle condition} if $|a-b|\leq c\leq a+b$ and $a+b+c$ is even. If $\charr\bk=0$, then $V_c$ is a summand of $V_a\otimes V_b$ if and only if $(a,b,c)$ satisfies the $0$-triangle condition. If $p$ is a prime, then we say that the triple $(a,b,c)\in\mN^3$ satisfies the \textbf{$p$-triangle condition} if:
\begin{enumerate}
\item $a,b,c\leq p-2$,
\item $(a,b,c)$ satisfies the 0-triangle condition, and
\item $a+b+c\leq 2p-4$.
\end{enumerate}
If $\charr\bk=p$ and $a,b,c\leq p-2$, then $V_c$ is a summand of $V_a\otimes V_b$ if and only if $(a,b,c)$ satisfies the $p$-triangle condition, and this summand has multiplicity 1. All other summands of $V_a\otimes V_b$ are tilting modules with highest weights at least $p-1$ (see \cite[Lemma~1.3]{DH}).

Now let $p=\charr\bk$ be either zero or prime. For $(a,b,c)$ satisfying the $p$-triangle condition, write $\incl^{a,b}_c$ for the inclusion $V_c\hookrightarrow V_a\otimes V_b$ sending the highest weight vector $v_{c,c}=x^c$ to
\[\incl^{a,b}_c(v_{c,c})=\sum_{r=0}^\chi (-1)^r \binom{\chi}{r} v_{a,a-r}\otimes v_{b,b-\chi+r}\]
where $\chi=(a+b-c)/2$. Similar inclusions $\iota^{a,b}_c$ are defined in \cite{sl2comb}, which are related to these by $\incl^{a,b}_c=\frac{(\chi+c)!}{a!b!}\iota^{a,b}_c$. We then define $\proj^{a,b}_c$ to be the projection $V_a\otimes V_b\twoheadrightarrow V_c$ such that $\proj^{a,b}_c\circ\incl^{a,b}_c=\id_{V_c}$. Complete explicit formulas for these morphisms are as follows:
\begin{align*}
\incl^{a,b}_c(v_{c,i})&=\sum_{r=i-c}^{\chi}\sum_{s=\max(0,r)}^{\min(\chi,c-i+r)}(-1)^s\frac{\binom{\chi}{s}\binom{a-s}{a-c+i-r}\binom{b-\chi+s}{b-\chi+r}}{\binom ci}v_{a,a-c+i-r}\otimes v_{b,b-\chi+r},\\
\proj^{a,b}_c(v_{a,i}\otimes v_{b,j})&=\frac{c+1}{\chi+c+1}\binom{c}{a-\chi}\left(\sum_{r=0}^\chi (-1)^r\frac{\binom{a-i}{r}\binom{b-j}{\chi-r}}{\binom{\chi+c}{a-r}}\right)v_{c,i+j-\chi}
\end{align*}
where $\proj_c^{a,b}(v_{a,i}\otimes v_{b,j})=0$ whenever $i+j-\chi<0$ or $i+j-\chi>c$. For $(a,b,c)$ not satisfying the $p$-triangle condition, we set $\incl^{a,b}_c=0$ and $\proj^{a,b}_c=0$. Note that for $p>0$ the coefficients in the formulas for $\incl^{a,b}_c$ and $\proj^{a,b}_c$ above are in $\mZ_{(p)}=\{\frac rs\mid r,s\in\mZ,p\nmid s\}$, and hence are well defined in $\bk$ via $\mZ_{(p)}\to\mF_p\hookrightarrow\bk$.

\subsection{The Lie algebra $\fsl_2$}\label{SecSL2Lie} For $p=\charr\bk\neq2$, the Lie algebra $\fsl_2$ over $\bk$ has a basis $\{e,f,h\}$ with the following action on $V_n$:
\[e\cdot v_i=(n-i)v_{i+1},\quad h\cdot v_i=(2i-n)v_i,\quad f\cdot v_i=iv_{i-1}.\]
We have an isomorphism of $\SL_2$-modules $\iota:\fsl_2\to V_2$ given by
\[e\mapsto2v_2,\quad h\mapsto-4v_1,\quad f\mapsto-2v_0.\]
This isomorphism is such that, for $p\geq5$ and $1\leq n\leq p-3$, the following diagrams commute:
\[\begin{tikzcd}
\fsl_2\otimes\fsl_2 \arrow{r}{\text{Lie bracket}} \arrow[swap]{d}{\iota\otimes\iota} &[3em] \fsl_2 \arrow{d}{\iota}\\
V_2\otimes V_2 \arrow{r}{\proj^{2,2}_2} & V_2
\end{tikzcd}\qquad\begin{tikzcd}
\fsl_2\otimes V_n \arrow{r}{\text{Lie action}} \arrow[swap]{d}{\iota\otimes1} &[3em] V_n \arrow{d}{\id}\\
V_2\otimes V_n \arrow{r}{\frac{n+2}4\proj^{2,n}_n} & V_n
\end{tikzcd}\]
Consequently, throughout this paper we will often identify $\fsl_2$ with $V_2$ via $\iota$, and identify the action of $\fsl_2$ on $V_n$ with $\frac{n+2}4\proj^{2,n}_n$. Note that if $n=0$, or if $p>0$ and $n=p-2$, then $\proj^{2,n}_n$ is automatically zero by the $p$-triangle condition.

\subsection{Associativity and braiding} Now that we have fixed direct sum components of $V_a\otimes V_b$, we can express the associativity and braiding constraints in $\Rep\SL_2$ in terms of these components. If $(a,b,c)$ satisfies the $p$-triangle condition, then we have
\begin{equation}\label{EqnBraidComposition}
\proj_c^{b,a}\circ \braid_{V_a,V_b}\circ\incl_c^{a,b}=(-1)^{(a+b-c)/2}\id_{V_c}.
\end{equation}
Similarly, for $a,b,c,r,s,d\in\mN$ such that $(a,b,r),(c,d,r),(a,d,s),(b,c,s)$ all satisfy the $p$-triangle condition, we can define a scalar $\alpha^{a,b,r}_{c,d,s}$ such that
\begin{equation}\label{EqnAssocComposition}
\proj_d^{a,s}\circ(\id_{V_a}\otimes\proj_s^{b,c})\circ(\incl_r^{a,b}\otimes\id_{V_c})\circ\incl_d^{r,c}=\alpha^{a,b,r}_{c,d,s}\id_{V_d}.
\end{equation}
This composition of morphisms is illustrated by the following string diagram:
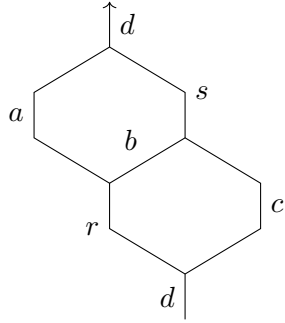
\begin{center}
\begin{tikzpicture}[yscale=0.6, xscale=1]
\draw (0,0) -- (0,1) node[left,pos=0.5] {$d$} -- (-1,2) node[left] {$r$} -- (-1,3) -- (-2,4) -- (-2,5) node[left,pos=0.5] {$a$} -- (-1,6);
\draw[->] (0,1) -- (1,2) -- (1,3) node[right,pos=0.5] {$c$} -- (0,4) -- (0,5) node[right] {$s$} -- (-1,6) -- (-1,7) node[right,pos=0.5] {$d$};
\draw (-1,3) -- (0,4) node[above left,pos=0.5] {$b$};
\end{tikzpicture}
\end{center}
If one of the triples above does not satisfy the $p$-triangle condition then we set $\alpha^{a,b,r}_{c,d,s}=0$. Using the results of \cite{sl2comb}, we have the following explicit formula for this scalar when the $p$-triangle condition is satisfied:
\begin{equation}\label{EqnAssociator}
\begin{split}
\alpha^{a,b,r}_{c,d,s}=(-1&)^{s+\frac{a-b-c+d}2}\frac{(s+1)!D(\frac a2,\frac b2,\frac r2)D(\frac c2,\frac d2,\frac r2)}{r!(\chi_1+1)!(\chi_2+1)!(\chi_3+1)!(\chi_4+1)!}\\
&\cdot\sum_{t=\chi^-}^{\chi^+}\frac{(-1)^t(t+1)!}{(t-\chi_1)!(t-\chi_2)!(t-\chi_3)!(t-\chi_4)!(\chi_5-t)!(\chi_6-t)!(\chi_7-t)!}
\end{split}
\end{equation}
where
\begin{align*}
\chi_1&=\frac{a+b+r}2,&\chi_2&=\frac{c+d+r}2,&\chi_3&=\frac{a+d+s}2,&\chi_4&=\frac{b+c+s}2,\\
\chi_5&=\frac{a+b+c+d}2,&\chi_6&=\frac{a+c+r+s}2,&\chi_7&=\frac{b+d+r+s}2,
\end{align*}
\[\chi^-=\max(\chi_1,\chi_2,\chi_3,\chi_4),\quad\chi^+=\min(\chi_5,\chi_6,\chi_7),\]
\[D(x,y,z)=(x+y+z+1)!(-x+y+z)!(x-y+z)!(x+y-z)!.\]
Once again, for $p>0$, Equation~(\ref{EqnAssociator}) is well-defined over $\mZ_{(p)}$ and hence gives a valid formula for $\alpha^{a,b,r}_{c,d,s}$ over $\bk$. However, note that Equation~(\ref{EqnAssociator}) only applies when all the triples above satisfy the $p$-triangle condition, and otherwise the right hand side may be non-zero in $\bk$ while $\alpha^{a,b,r}_{c,d,s}=0$.

When $\bk=\mC$, the scalar $\alpha^{a,b,r}_{c,d,s}$ is related to the well-known Racah-Wigner $6j$-symbol as follows:
\begin{equation}\label{Eqn6j}
\alpha^{a,b,r}_{c,d,s}=(-1)^{s+\frac{a-b-c+d}2}\frac{(s+1)!}{r!}\sqrt{\frac{D(\frac a2,\frac b2,\frac r2)D(\frac c2,\frac d2,\frac r2)}{D(\frac a2,\frac d2,\frac s2)D(\frac b2,\frac c2,\frac s2)}}\left\{\begin{matrix}\frac a2&\frac b2&\frac r2\\\frac c2&\frac d2&\frac s2\end{matrix}\right\}.
\end{equation}
The $6j$-symbol is invariant under permuting columns, and also under swapping the upper and lower entries in two columns simultaneously. From this we observe the relations
\begin{equation}\label{EqnPerms}
\begin{split}
\alpha^{a,b,r}_{c,d,s}&=\alpha^{b,a,r}_{d,c,s}=\alpha^{c,d,r}_{a,b,s}=\alpha^{d,c,r}_{b,a,s},\\
\alpha^{a,r,b}_{c,s,d}&=(-1)^{\frac{b+d-r-s}2}\frac{(d+1)!r!}{b!(s+1)!}\frac{D(\frac b2,\frac c2,\frac s2)}{D(\frac c2,\frac d2,\frac r2)}\alpha^{a,b,r}_{c,d,s}
\end{split}
\end{equation}
which hold over arbitrary $\bk$.

Note that we can compute the associator in the opposite direction (that is, bracketing on the right to bracketing on the left) by evaluating the composition
\[\begin{split}
V_a\otimes (V_b\otimes V_c)&\xrightarrow{\braid}(V_b\otimes V_c)\otimes V_a\xrightarrow{\braid}(V_c\otimes V_b)\otimes V_a\\
&\cong V_c\otimes(V_b\otimes V_a)\xrightarrow{\braid}V_c\otimes(V_a\otimes V_b)\xrightarrow{\braid}(V_a\otimes V_b)\otimes V_c
\end{split}\]
restricted to $V_d\hookrightarrow V_a\otimes V_s$ and $V_r\otimes V_c\twoheadrightarrow V_d$. This involves braidings of $V_k\subseteq V_i\otimes V_j$ for $(i,j,k)$ equal to $(a,s,d),(b,c,s),(b,a,r),(c,r,d)$. Since
\[\frac{a+s-d}2+\frac{b+c-s}2+\frac{b+a-r}2+\frac{c+r-d}2=(a+b+r)+(c-d-r)\]
is even when these triples satisfy the $p$-triangle condition, we have
\begin{equation}\label{EqnRightLeft}
\begin{split}
\proj_d^{r,c}&\circ(\proj_r^{a,b}\otimes\id_{V_c})\circ(\id_{V_a}\otimes\incl_s^{b,c})\circ\incl_d^{a,s}\\
&=\alpha^{c,b,s}_{a,d,r}\id_{V_d}=(-1)^{-b+d-r+s}\frac{r!(r+1)!}{s!(s+1)!}\frac{D(\frac a2,\frac d2,\frac s2)D(\frac b2,\frac c2,\frac s2)}{D(\frac a2,\frac b2,\frac r2)D(\frac c2,\frac d2,\frac r2)}\alpha^{a,b,r}_{c,d,s}\id_{V_d}
\end{split}
\end{equation}
using the $6j$-symbol description again.

\subsection{Rational functions}\label{SecRational} The number of terms in the sum in Equation~(\ref{EqnAssociator}) is $\chi^+-\chi^-+1$, unless this value is zero or negative in which case the associator is zero. Notice that if $a=b$ or $c=d$, then the $p$-triangle condition forces $r$ to be even, and $\chi^+-\chi^-\leq\frac r2$. We will use this to our advantage to obtain a more useful expression for $\alpha^{a,b,r}_{c,d,s}$ when $a=b=2x$, $c=d=2y$, $s=2z$, and $r=2u$ is small.

\begin{prop}\label{PropRational}
Let $p=\charr\bk\geq3$. If $u,x,y,z\in\{0,\dots,\frac{p-3}2\}$ are such that $|x-y|\leq z\leq x+y$ and $0\leq u\leq\min(2x,2y)$, then we have $\alpha^{2x,2x,2u}_{2y,2y,2z}=f_u(x,y,z)g(x,y,z)$
where
\[g(x,y,z)=\begin{cases}\frac{(-1)^{x+y+z}(2x)!(2y)!(2z)!}{(x+y+z+1)!(-x+y+z)!(x-y+z)!(x+y-z)!}&\text{ if }x+y+z\leq p-2\\0&\text{ if }x+y+z\geq p-1\end{cases},\]
\[\begin{split}f_u(x,y,z)=\frac{2z+1}{\binom{2u}{u}(2x)_u(2y)_u}\sum_{t=0}^u(-1)^t&\binom{u}{t}^2(x+y+z+1+t)_t\\
&\cdot(-x+y+z)_{u-t}(x-y+z)_{u-t}(x+y-z)_t\end{split}\]
and $(a)_b=a(a-1)(a-2)\cdots(a-b+1)$ for $a\in\mZ$ and $b\in\mN$ (with the convention $(a)_0=1$).
\end{prop}

\begin{proof}
We may assume $(2x,2y,2z)$ satisfies the $p$-triangle condition, that is $x+y+z\leq p-2$, since otherwise $g(x,y,z)=0=\alpha^{2x,2x,2u}_{2y,2y,2z}$. If a term in the sum defining $f_u(x,y,z)$ with parameter $t$ is non-zero, then
\[x+y+z+1+t\leq p-1,\quad -x+y+z-u+t\geq0,\quad x-y+z-u+t\geq0.\]
Subtracting the second and third equations from the first gives $u+2x,u+2y\leq p-2$. Consequently, we may also assume that $(2u,2x,2x)$ and $(2u,2y,2y)$ satisfy the $p$-triangle condition, since otherwise $f_u(x,y,z)=0=\alpha^{2x,2x,2u}_{2y,2y,2z}$. Hence we can use Equation~(\ref{EqnAssociator}), in which reparametrising $t$ to $t+x+y+z$ gives the required equality. Note that this would make the sum start at $t=\max(0,u+x-y-z,u-x+y-z)$ and end at $t=\min(x+y-z,u)$, but for $t$ outside this range the corresponding summand in $f_u(x,y,z)$ is automatically zero, so we may instead sum from $0$ to $u$.
\end{proof}

The expressions $f_u(x,y,z)$ and $g(x,y,z)$ have the following important properties:
\begin{enumerate}
\item $f_u(x,y,z)$ is invariant under swapping $x$ and $y$,
\item $g(x,y,z)$ is invariant under all permutations of $x,y,z$,
\item $g(x,y,z)=0$ if and only if $x+y+z\geq p-1$, and
\item for any fixed $u$, $f_u(x,y,z)$ is a rational function in $x,y,z$, and the denominator of this rational function is not divisible by $p$ for any $x,y$ with $\frac u2\leq x,y\leq\frac{p-3}2$.
\end{enumerate}
For example, $f_1(x,y,z)=(2z+1)(-x^2-x-y^2-y+z^2+z)/(4xy)$.

\section{The semisimple Verlinde category and Lie algebras}\label{SecLieAlgs}

In this section we recall the definitions of Lie algebras in symmetric tensor categories. We also introduce based Lie algebras, and obtain a characterisation of them in $\Ver_p$.

\subsection{Tensor categories and $\Ver_p$} We call a $\bk$-linear braided monoidal category $\cC$ a \textbf{symmetric tensor category} if $\cC$ is essentially small and abelian, every object in $\cC$ has finite length, the monoidal structure $\otimes$ is rigid, the unit object $\unit$ satisfies $\End(\unit)\cong\bk$, and the braiding satisfies $\braid_{Y,X}\circ\braid_{X,Y}=\id_{X\otimes Y}$ for all $X,Y\in\cC$. We call a braided monoidal functor between symmetric tensor categories a \textbf{symmetric tensor functor} if it is exact and faithful. We write $\Indd(\cC)$ for the ind-completion of $\cC$, which inherits a monoidal structure from $\cC$. In particular, we write $\Vecc_\bk$ and $\sVec_\bk$ for finite-dimensional vector spaces and super vector spaces respectively over $\bk$, and then $\Indd(\Vecc_\bk)$ and $\Indd(\sVec_\bk)$ are the categories of all vector spaces and super vector spaces.

Suppose $p=\charr\bk\geq2$. The semisimple Verlinde category $\Ver_p$ is the semisimplification of the category of tilting modules $\Tilt\SL_2$. The simple objects in $\Ver_p$ are $L_i$, the image of $V_i$, for $i\in\{0,\dots,p-2\}$ (note that some papers such as \cite{asymptotic} use the convention $L_i=V_{i-1}$). We abuse notation and write $\incl^{a,b}_c$ and $\proj^{a,b}_c$ for the images of these maps in $\Ver_p$ as well. Whenever $L_c,L_{c'}$ are subobjects of $L_a\otimes L_b$, we have $\proj_{c'}^{a,b}\circ\incl_c^{a,b}=\id_{L_c}$ for $c=c'$ and 0 otherwise, and hence these morphisms give a complete direct sum decomposition of $L_a\otimes L_b$. Since the functor $\Tilt\SL_2\to\Ver_p$ is braided monoidal, Equation~(\ref{EqnBraidComposition}) holds in $\Ver_p$ whenever $(a,b,c)$ satisfy the $p$-triangle condition, and Equations~(\ref{EqnAssocComposition}) and (\ref{EqnAssociator}) hold whenever $(a,b,r)$, $(c,d,r)$, $(a,d,s)$, $(b,c,s)$ satisfy the $p$-triangle condition.

For $p\geq5$, $\Ver_p$ has three proper tensor subcategories: $\Ver_p^+$, which has simples $L_{2n}$ for $n\in\{0,\dots,\frac{p-3}2\}$; a subcategory equivalent to $\sVec$, which has simples $\unit=L_0$ and $\bar\unit=L_{p-2}$; and a subcategory equivalent to $\Vecc$ with the simple $\unit$.

\subsection{Lie algebras and operadic Lie algebras}\label{SecLieDef} Following \cite{Et}, an \textbf{operadic Lie algebra} in a symmetric tensor category $\cD$ is an (ind-)object $\fg$ along with a morphism $[,]:\fg\otimes\fg\to\fg$ called the Lie bracket satisfying:
\begin{enumerate}
\item antisymmetry, meaning $[,]\circ\braid_{\fg,\fg}=-[,]$; and
\item the Jacobi identity, meaning $[[,],]\circ(\id_{\fg^{\otimes3}}+\sigma+\sigma^2)=0$ where $\sigma:\fg^{\otimes3}\to\fg^{\otimes 3}$ is a cyclic braiding, $\sigma^3=\id_{\fg^{\otimes3}}$. 
\end{enumerate}
A Lie algebra morphism $\fg\to\fg'$ is a morphism in $\Indd(\cD)$ which commutes with the bracket maps on $\fg$ and $\fg'$. Quotients of operadic Lie algebras $\fg\twoheadrightarrow\fg'$ are in correspondence with ideals of $\fg$, which are subobjects $I\subseteq\fg$ such that the image of the restriction of $[,]$ to $\fg\otimes I$ is contained in $I$. We say that $\fg$ is \textbf{simple} if its only ideals are $0$ and $\fg$, and $[,]$ is non-zero. We write $\fg_0$ for the underlying ordinary part of $\fg$, which is the Lie subalgebra spanned by subobjects isomorphic to $\unit$, and we call $\fg$ \textbf{invariantless} if $\fg_0=0$. A representation of $\fg$ is an object $X\in\cD$ with a Lie algebra morphism $\fg\to\fgl(X)=X\otimes X^*$; this gives a corresponding action map $\fg\otimes X\to X$, and a morphism of $\fg$-representations is a morphism in $\cD$ which commutes with the action maps. We write $\Rep_\cD(\fg)$ or just $\Rep(\fg)$ for the category of representations of $\fg$.

If $A$ is an associative algebra in $\Indd(\cD)$ with multiplication $\nabla:A\otimes A\to A$, then $A$ has a natural operadic Lie algebra structure with Lie bracket $\nabla\circ(1-\braid_{A,A})$. Operadic Lie algebras $\fg$ coming from associative algebras are more constrained than general operadic Lie algebras; for example over $\Vecc$ we have $[v,v]=0$ for $v\in\fg$ even though this is not guaranteed by antisymmetry in characteristic 2, and similarly over $\sVec$ we have $[[v,v],v]=0$ for odd $v$ even though operadic Lie algebras may not satisfy this in characteristic 3. In order to capture these constraints, we recall the following definition from \cite{Et}. The forgetful functor from operadic Lie algebras to $\Indd(\cD)$ has a left adjoint sending $X\in\Indd(\cD)$ to the free operadic Lie algebra ${\rm FOLie}(X)$, and we have a Lie algebra morphism from ${\rm FOLie}(X)$ to the tensor algebra $T(X)$ with kernel denoted $E(X)$. An operadic Lie algebra $\fg$ is called a \textbf{Lie algebra} if the adjunction counit map ${\rm FOLie}(\fg)\to\fg$ sends $E(\fg)$ to zero.

\begin{prop}\label{PropNoL1}
Let $p\geq5$ be a prime, and let $\fg$ be an operadic Lie algebra in $\Ver_p$.
\begin{enumerate}
\item $\fg$ is a Lie algebra if and only if it is a Lie subalgebra of some associative algebra.
\item If either no summand of $\fg$ is isomorphic to $L_1$, or no summand of $\fg$ is isomorphic to $\bar\unit$, then $\fg$ is a Lie algebra.
\end{enumerate}
\end{prop}

\begin{proof}
For (1), one can define the universal enveloping algebra $U(\fg)$ of an operadic Lie algebra $\fg$, and \cite[Theorem~6.6]{Et} shows that if $\fg$ is a Lie algebra then it is PBW, which implies that the map $\fg\to U(\fg)$ is injective. The converse is \cite[Proposition~4.7]{Et}.

For (2), for an operadic Lie algebra $\fg$ in $\Ver_p$ one can construct an additive map $\gamma_p:\Hom_{\Ver_p}(L_1,\fg)\to\Hom_{\Ver_p}(\bar\unit,\fg)$ sending $f:L_1\to\fg$ to a composition $\bar\unit\hookrightarrow L_1^{\otimes p}\to\fg^{\otimes p}\to\fg$, where the last map is a sequence of Lie brackets. By \cite[Theorem~6.6]{Et}, $\fg$ is a Lie algebra if and only if $\gamma_p=0$.
\end{proof}

For an object $X\in\cD$, $X\otimes X^*$ has an associative algebra structure with multiplication $\id_X\otimes\ev_X\otimes\id_{X^*}$, and we write $\fgl(X)$ for the associated Lie algebra structure. Following \cite{Ve2}, we also write $\fsl(X)$ for the Lie subalgebra of $\fgl(X)$ given by the kernel of the map $\ev_X\circ\braid_{X,X^*}$.

Since $\sVec$ is a subcategory of $\Ver_p$ for $p\geq3$, we give the following elementary lemma about ordinary Lie superalgebras, or in other words Lie algebras in $\sVec$, which is an edge case of Theorem~\ref{ThmLength3}.

\begin{lemma}\label{LemLieSuper}
Let $\bk$ be algebraically closed, and let $\charr\bk$ be either zero or $p\geq5$. If $\fg$ is a Lie superalgebra over $\bk$ with graded dimension either $(n|1)$ or $(1|n)$ for some integer $n\geq1$, then $\fg$ is non-simple. The only simple (non-abelian) Lie superalgebra with total dimension at most 3 is $\fsl_2$.
\end{lemma}

\begin{proof}
If $\fg$ has dimension $(n|1)$, then let $v\in\fg$ span the odd part. If $[v,v]=0$ then $v$ spans an ideal. Otherwise $[v,v]$ spans an ideal, since the Jacobi identity shows that $[[v,v],v]=0$ and $[[v,v],u]$ is a multiple of $[v,v]$ for all even $u\in\fg$.

If $\fg$ has dimension $(1|n)$, then let $u\in\fg$ span the even part and write $V\subset\fg$ for the odd part. Let $R=\{r\in V\mid [r,v]=0\text{ for all }v\in V\}$ be the radical of the bracket map $V\otimes V\to\bk$. If $R\neq0$, then $[[u,r],v]=0$ for $r\in R$ and $v\in V$ by the Jacobi identity, so $[u,r]\in R$. Thus $R$ is an ideal and $\fg$ is non-simple. If instead $R=0$, then since $[[v,v],v]=0$ by the Jacobi identity, the map $[u,\cdot]:V\to V$ is zero on vectors $v\in V$ with $[v,v]\neq0$. But such vectors span $V$ (e.g. see \cite[Theorem~6.9]{Jac}), so this map is zero and $u$ spans an ideal.

For Lie superalgebras up to dimension 3, the only cases not covered above are purely odd (and thus abelian) or purely even, for which the result follows from classical Lie theory for characteristic zero and \cite{PS} for positive characteristic.
\end{proof}

\subsection{Harish-Chandra pairs and restricted Lie algebras} An \textbf{affine group scheme} $G$ in a symmetric tensor category $\cD$ is a representable functor from the category of commutative algebras in $\Indd(\cD)$ to the category of groups. The functor $G$ is represented by a Hopf algebra $\cO(G)$ in $\Indd(\cD)$, and a representation of $G$ is an object $X\in\cD$ with a coaction of $\cO(G)$. We write $\Lie(G)$ for the Lie algebra in $\Indd(\cD)$ corresponding to $G$; see \cite{BP} for a number of equivalent formalisations of this.

Recall from \cite{Ve1} the notion of a \textbf{Harish-Chandra pair} in $\Ver_p$, which is a pair $(G_0,\fg)$ of a finite-type affine group scheme $G_0$ over $\Vecc$ and a finite-length Lie algebra $\fg$ in $\Ver_p$, along with an action of $G_0$ on $\fg$ and an isomorphism $i:\Lie(G_0)\cong\fg_0$, such that $i$ and the bracket on $\fg$ are $G_0$-module morphisms and the actions of $\fg_0$ and $\Lie(G_0)$ on $\fg$ are identified by $i$. By the results of \cite{Ve1}, we have a category equivalence between finite-type affine group schemes over $\Ver_p$ and Harish-Chandra pairs. Moreover, by the results of \cite{Ve2}, an action of $G$ on an object $X=\bigoplus_i L_i^{\oplus n_i}\in\Ver_p$ corresponds to a group morphism $G\to\GL(X)$ where $\GL(X)$ is an affine group scheme with Harish-Chandra pair $(\prod_i\GL_{n_i},\fgl(X))$.

Following \cite{BP}, a \textbf{restricted Lie algebra} structure on an (ind-)object $\fg$ in a symmetric tensor category $\cD$ is given by a collection of maps $\phi_n:(\mathcal Lie_n\otimes\fg^{\otimes n})^{S_n}\to\fg$, where $\mathcal Lie_n$ is the degree $n$ part of the Lie operad, the symmetric group $S_n$ acts simultaneously by permutations on the monomial basis of $\mathcal Lie_n$ and by braidings on $\fg^{\otimes n}$, and $(-)^{S_n}$ refers to invariants of $S_n$. These maps satisfy conditions coming from the operad structure on $\mathcal Lie$. If $\cD=\Vecc_\bk$ with $\charr\bk=p$, then a restricted Lie algebra is exactly a Lie algebra along with a ``$p$-th power map'' $v\mapsto v^{[p]}$ satisfying certain conditions. Not all Lie algebras have such a map, and this map is not uniquely determined by the Lie algebra structure (see for instance \cite{Bo}).

\begin{prop}[Theorem~12.6 in \cite{BP}]\label{PropLieRes}
If $G$ is an affine group scheme in $\cD$, then $\Lie(G)$ has a natural restricted Lie algebra structure.
\end{prop}

\begin{corollary}\label{CorLieRes}
If $\fg$ is an invariantless Lie algebra in $\Ver_p$ with $p\geq5$, then $(1,\fg)$ is a Harish-Chandra pair, and the corresponding affine group scheme $G$ gives $\fg\cong\Lie(G)$ a natural restricted Lie algebra structure.
\end{corollary}

Because of this corollary, and also Proposition~\ref{PropNoL1}, almost all operadic Lie algebras we will consider in this paper are actually restricted Lie algebras.

\subsection{$\cD$-groups and based Lie algebras} Let $\cD$ be a symmetric tensor category, and recall the following notions from \cite{incompressible}. The fundamental group $\pi(\cD)=\underline{\Aut}^\otimes(\Id_\cD)$ is an affine group scheme over $\cD$ which has a canonical action on all objects of $\cD$. A \textbf{$\cD$-group} is defined as a pair $(G,\phi)$ of an affine group scheme $G$ in $\cD$ and a morphism $\phi:\pi(\cD)\to G$ such that the conjugation action of $\pi(\cD)$ on $G$ induces the canonical action of $\pi(\cD)$ on $\cO(G)$ as an ind-object in $\cD$. By \cite[Theorem~4.2.1]{incompressible} we have a contravariant equivalence between $\cD$-groups and pairs $(\cC,F)$ of a symmetric tensor category $\cC$ and symmetric tensor functor $F:\cC\to\cD$. This equivalence sends $(G,\phi)$ to $\Rep(G,\phi)$, the category of representations of $G$ for which the action of $\pi(\cD)$ via $\phi$ is the canonical action.

Write $\fg_\pi=\Lie(\pi(\cD))$. The canonical action of $\pi(\cD)$ induces an action of $\fg_\pi$ on objects in $\cD$. We define a \textbf{based Lie algebra} in $\cD$ to be a pair $(\fg,\phi)$ of a Lie algebra $\fg\in\Indd(\cD)$ and a Lie algebra morphism $\phi:\fg_\pi\to\fg$ such that the adjoint action of $\fg_\pi$ on $\fg$ via $\phi$ is the canonical action on the underlying object of $\fg$ (\cite{EKO} uses the term ``anchored'' for the same concept in the case of reduced KKT Lie algebras, but we use a different term to avoid connotations with anchored vector bundles and other related objects). We write $\Rep(\fg,\phi)$ for the category of representations of $\fg$ for which the action of $\fg_\pi$ via $\phi$ is the canonical action. If $(G,\phi)$ is a $\cD$-group then $(\fg,\Lie(\phi))$ is a based Lie algebra, and we have a restriction functor $\Rep(G,\phi)\to\Rep(\fg,\Lie(\phi))$.

Proposition~\ref{PropCanAction} below gives a characterisation of based Lie algebras in $\Ver_p$ for $p\geq5$. First, recall from \cite[Theorem~4.9]{Ve2} that $\pi(\Ver_p^+)$ has Harish-Chandra pair $(1,\fsl(L_1))$. Then the Harish-Chandra pair for $\pi(\Ver_p)$ is $(\mZ/2,\fsl(L_1))$ with the trivial action of $\mZ/2$ on $\fsl(L_1)$, since $\Ver_p\simeq\Ver_p^+\boxtimes\sVec$ and $\pi(\sVec)\cong\mZ/2$. Now, since $\fsl(L_1)$ is the image in $\Ver_p$ of $\fsl(V_1)$ in $\Rep\SL_2$, which in turn is isomorphic to $\fsl_2$, the isomorphism $\iota$ in Section~\ref{SecSL2Lie} shows that $\fsl(L_1)$ is isomorphic to $L_2$ with the Lie bracket $\proj^{2,2}_2$.

\begin{lemma}\label{LemUniqueAction}
Let $p\geq5$ be a prime, and let $\fg=L_2$ be the Lie algebra with bracket $\proj^{2,2}_2$. For $n\in\{0,\dots,p-2\}$, a morphism $f:L_2\otimes L_n\to L_n$ is an action of $\fg$ on $L_n$ if and only if either $f=\frac{n+2}4\proj^{2,n}_n$ or $f=0$. Note that $\proj^{2,n}_n$ is also zero for $n\in\{0,p-2\}$.
\end{lemma}

\begin{proof}
For $n\in\{0,p-2\}$ the only morphism is $f=0$. Otherwise, $\dim\Hom(L_2\otimes L_n,L_n)=1$, so $f=\lambda\frac{n+2}4\proj^{2,n}_n$ for some $\lambda\in\bk$. This is a Lie algebra action if and only if
\[f\circ(1\otimes f)\circ((1-\braid_{L_2,L_2})\otimes1)=f\circ([,]\otimes1)\]
as morphisms $L_2\otimes L_2\otimes L_n\to L_n$. If $\lambda=1$ then $f$ is the image in $\Ver_p$ of the ordinary Lie algebra action $\fsl_2\otimes V_n\to V_n$ as shown in Section~\ref{SecSL2Lie}, so the equation above holds and $f\circ([,]\otimes1)\neq0$. This means that the equation is equivalent to $\lambda^2=\lambda$, so $\lambda\in\{0,1\}$.
\end{proof}

\begin{prop}\label{PropCanAction}
Let $p\geq5$ be a prime. An action of $\pi(\Ver_p)$ on an object $X\in\Ver_p$ induces actions of $\mZ/2$ and $\fsl(L_1)$ on $X$. This is the canonical action if and only if it diagonalises over a decomposition of $X$ into simples such that:
\begin{enumerate}
\item the action of $\mZ/2$ on each simple summand $L_n\subseteq X$ is the parity action, meaning $1\in\mZ/2$ acts by $(-1)^n$, and
\item the action of $\fsl(L_1)\cong L_2$ on each simple summand $L_n\subseteq X$ with $n\not\in\{0,p-2\}$ is the unique non-trivial action as in Lemma~\ref{LemUniqueAction}.
\end{enumerate}
In particular, if $\fg$ is a Lie algebra and $\phi:\fsl(L_1)\to\fg$ is a Lie algebra morphism, then $(\fg,\phi)$ is a based Lie algebra if and only if the adjoint action of $\fsl(L_1)$ on $\fg$ via $\phi$ satisfies (2).
\end{prop}

\begin{proof}
Because of the correspondence between actions of $\pi(\Ver_p)$ on $X$ and maps $\pi(\Ver_p)\to\GL(X)$, two actions of $\pi(\Ver_p)$ are equal if and only if the corresponding actions of $\mZ/2$ and $\fsl(L_1)$ are equal. So the proposition holds if the canonical action diagonalises and induces the actions (1) and (2), and this follows from $\Ver_p\simeq\Ver_p^+\boxtimes\sVec$ and Lemma~\ref{LemUniqueAction}.
\end{proof}

We will often omit $\phi$ for based Lie algebras in $\Ver_p$, because of the following fact:

\begin{prop}\label{PropUnique}
If $\fg$ is a Lie algebra in $\Ver_p$ with $p\geq5$, then there is at most one Lie algebra morphism $\phi:\fsl(L_1)\to\fg$ such that $(\fg,\phi)$ is a based Lie algebra. If this morphism exists, then it is zero if and only if $\fg\in\sVec\subset\Ver_p$.
\end{prop}

\begin{proof}
Suppose such a morphism $\phi$ exists. If $\fg\in\sVec$, then the only $\Ver_p$-morphism $\fsl(L_1)=L_2\to\fg$ is the zero morphism, and we are done. If instead there is a simple summand $L_n\subseteq\fg$ with $n\in\{1,\dots,p-3\}$, then the canonical action of $\fsl(L_1)$ on $L_n$ is non-zero, so $\phi$ must be non-zero as well. Now suppose for a contradiction that there are two non-zero maps $\phi_1,\phi_2:\fsl(L_1)\to\fg$ both giving the canonical action of $\fsl(L_1)$. If $\phi_2=\lambda\phi_1$ for some $\lambda\in\bk^\times$, then
\[\phi_1\circ[,]=[,]\circ(\phi_1\otimes\phi_1)\text{ and }\phi_2\circ[,]=[,]\circ(\phi_2\otimes\phi_2)\]
imply that $\lambda^2=\lambda$, hence $\lambda=1$ and $\phi_1=\phi_2$. Otherwise, we have a direct sum decomposition $\fg\cong \im(\phi_1)\oplus \im(\phi_2)\oplus\fg'$, and the Lie bracket restricted to $\im(\phi_1)\otimes \im(\phi_2)\to\im(\phi_1)$ must be zero for $\phi_1$ to be canonical and non-zero for $\phi_2$ to be canonical, a contradiction.
\end{proof}

Recall from \cite{De} that if $\charr\bk\neq2$, $G$ is a supergroup and $\phi:\mZ/2\to G$ is a group morphism, then $(G,\phi)$ is an $\sVec$-group if and only if $\phi$ induces the parity automorphism of $G$. On the Harish-Chandra pair $(G_0,\fg)$, this is equivalent to the image $\phi$ being central in $G_0$ and inducing the parity action of $\mZ/2$ on $\fg$. The following gives a similar characterisation of $\Ver_p$-groups:

\begin{prop}\label{PropBased}
Let $p\geq5$ be a prime. Let $G$ be a finite-type affine group scheme over $\Ver_p$ with Harish-Chandra pair $(G_0,\fg)$, and let $\phi:\pi(\Ver_p)\to G$ be a group morphism. Then $(G,\phi)$ is a $\Ver_p$-group if and only if:
\begin{enumerate}
\item $(\fg,\Lie(\phi))$ is a based Lie algebra;
\item the image of the map $\phi_0:\mZ/2\to G_0$ induced by $\phi$ is central in $G_0$; and
\item the action of $\mZ/2$ on $\fg$ coming from $\phi_0$ and the action of $G_0$ is the parity action.
\end{enumerate}
\end{prop}

\begin{proof}
We need to determine for which $\phi$ the action of $\pi(\Ver_p)$ on $\cO(G)$ is canonical. Since $\cO(G)$ is finite type, it is determined by its finite dual $H=\cO(G)^\circ$ by \cite[Proposition~4.33]{Ve1}. By \cite[Lemma~6.6]{Ve1} we have $H\cong U(\fg)\otimes_{U(\fg_0)}\cO(G_0)^\circ$ with $\fg$ and $\cO(G_0)^\circ$ embedding into $H$. This means that the action of $\pi(\Ver_p)$ on $H$ is canonical if and only if it is canonical on $\fg$ and $\cO(G_0)$. The action on $\fg$ gives conditions (1) and (3) by Proposition~\ref{PropCanAction}. Since all simple summands of $\cO(G_0)$ are isomorphic to $\unit$, the action of $\fsl(L_1)$ on $\cO(G_0)$ is zero and automatically canonical, while the action of $\mZ/2$ is canonical if and only if it is trivial. This is equivalent to the conjugation action of $\mZ/2$ on $G_0$ via $\phi_0$ being trivial, which is equivalent to the image of $\phi_0$ being central, giving condition (2).
\end{proof}

\begin{corollary}\label{CorBased}
Let $(\fg,\phi)$ be a based Lie algebra in $\Ver_p$ with $p\geq5$. If $\fg$ is invariantless, then $(\mZ/2,\fg)$ is a Harish-Chandra pair with the parity action of $\mZ/2$ on $\fg$, and the maps $\id_{\mZ/2}$ and $\phi$ give the corresponding affine group scheme a $\Ver_p$-group structure.
\end{corollary}

\begin{remark}\label{RemBased}
There exist Lie algebras in $\Ver_p$ which are not based, for example:
\begin{enumerate}
\item For $n\in\{1,\dots,p-3\}$, $L_n$ with the trivial bracket $[,]=0$ is a non-based Lie algebra.
\item The Lie superalgebra $\mathfrak{psl}(2|2)$ (Kac type $A(1,1)$) is simple in characteristic $p\geq3$ by \cite[Lemma~2.7]{BZ}, and has an outer action of $\SL_2$ under which it has even part $V_0^{\oplus 6}$ and odd part $V_1^{\oplus 4}$ (this is stated in \cite[\S6.4]{GP} over $\mC$, but the same methods apply in characteristic $p\geq3$). Thus, it is a superalgebra in $\Tilt\SL_2$, and its image in $\Ver_p$ can be turned into an operadic Lie algebra by tensoring the odd part by $\bar\unit$, as discussed in Section~\ref{SecSuper}. By Propositions~\ref{PropNoL1} and \ref{PropCanAction}, this is a non-based simple Lie algebra structure on the object $L_0^{\oplus6}\oplus L_{p-3}^{\oplus4}\in\Ver_p$ for all $p\geq5$. In particular, this only contains $L_2$ summands for $p=5$, and then the superalgebra grading means that no such summand can act canonically on itself.
\end{enumerate}
However, if $\fg$ is a Lie algebra in $\Ver_p$ with $p\geq5$, then we can construct a based Lie algebra $\widehat{\fg}=\fsl(L_1)\ltimes\fg$; that is, the underlying object is $\fsl(L_1)\oplus\fg$, and the bracket comes from the Lie algebra structures and the canonical action of $\fsl(L_1)$ on $\fg$. If $\phi$ is the inclusion $\fsl(L_1)\hookrightarrow\widehat{\fg}$ then we have an equivalence $\Rep(\fg)\simeq\Rep(\widehat{\fg},\phi)$ where a representation of $\fg$ is endowed with the canonical action of $\fsl(L_1)$ to turn it into a $\widehat{\fg}$-representation. This can also be performed on the level of groups; see \cite[Example~0.4(ii)]{De} for the corresponding construction in $\sVec$. Thus, for the purposes of understanding representations of Lie algebras in $\Ver_p$, it is sufficient to study only based Lie algebras. However, note that if $\fg$ is non-zero then $\widehat{\fg}$ is non-simple, since $\fg$ is an ideal of $\widehat{\fg}$.

Additionally, there exist based Lie algebras in $\Ver_p$ which do not integrate to an affine group scheme. In particular, some modular Lie algebras over $\bk$ do not have a restricted structure (for example the generalised Jacobson-Witt algebra $W(1;n)$ which is simple and non-restrictable for $n\geq2$ and $p>2$ by \cite[Ch.~4 Theorem~2.4]{SF}), and thus cannot be equal to $\Lie(G)$ for any group $G$ when interpreted in $\Vecc\subseteq\Ver_p$.
\end{remark}

\section{Images of ordinary Lie algebras and superalgebras in $\Ver_p$}\label{SecImages}

In this section, we discuss methods of constructing Lie algebras in $\Ver_p$ from ordinary Lie algebras and superalgebras, and prove part (b) of Theorem~\ref{ThmMain}. We also discuss the unusual Lie algebras arising from $E_8(a_1)$, $Q(2)$ and $H(5)$.

\subsection{Restrictions of Lie algebras to $\Ver_p$}\label{SecOrdinary} Recall the following construction from \cite[\S3.2]{Ka} which produces operadic Lie algebras in $\Ver_p$. For $\charr\bk=p>0$, we have an affine group scheme $\alpha_p$ over $\bk$ whose value on a commutative algebra $A$ is $\alpha_p(A)=\{a\in A\mid a^p=0\}$ with the addition operation. The group morphism $\alpha_p\to\SL_2$ given by $a\mapsto\sM{1&a\\0&1}$ induces an equivalence between $\Ver_p$ and the semisimplification of $\Rep_{\bk}\alpha_p\simeq\Mod(\bk[t]/t^p)$. If $\fg$ is an operadic Lie algebra in $\Vecc_\bk$ and $e$ is an element of $\fg$ such that $({\rm ad}(e))^p=0$, then $\fg$ has a $\Rep\alpha_p$-module structure where the action of $t\in\bk[t]/t^p$ is ${\rm ad}(e)$. Moreover, the bracket morphism is an $\alpha_p$-module morphism, and thus the image of $\fg$ and its bracket under $\Rep\alpha_p\to\Ver_p$ is an operadic Lie algebra.

A special case of this is where $G$ is a quasi-simple algebraic group and $e\in\Lie(G)$ is the image of $\sM{0&1\\0&0}$ under a map $\fsl_2\to\Lie(G)$ induced by a morphism $\phi:\SL_2\to G$ so that the restriction of $\Lie(G)$ to $\SL_2$ is a tilting module. When this is the case, we write $\Lie_p(G,\phi)$ for the image of $\Lie(G)$ in $\Ver_p$, and this is a based Lie algebra when equipped with the image of $\Lie(\phi)$. A sufficient condition for this construction to be well-defined is that $\Lie(G)$ is a tilting module of $G$ and $(G,\SL_2)$ with the map $\phi$ is a Donkin pair in the sense of \cite{HM} (for example this holds when $\phi$ is optimal and $G$ has classical type by \cite[Theorem~4.1.2]{HM}). 

Recall that a \textbf{principal morphism} $\SL_2\to G$ is a morphism whose image contains a regular unipotent element. If the map $\phi$ is principal, then we will abbreviate $\Lie_p(G,\phi)$ as $\Lie_p(G)$, and this is well-defined for $p>h$ where $h$ is the Coxeter number of $G$ by \cite[Lemma~3.2.2]{CEN}. Following \cite{EO}, we say that an inclusion of algebraic groups $K\subset G$ is a \textbf{principal pair} if $K$ contains a regular unipotent element of $G$; this means that there is a principal morphism $\SL_2\to K$ such that the composition $\SL_2\to K\to G$ is also principal, and thus the inclusion $\Lie(K)\subseteq\Lie(G)$ induces an inclusion $\Lie_p(K)\subseteq\Lie_p(G)$ in $\Ver_p$.

For $p>h$, define $\Ver_p(G)$ to be the semisimplification of the category of tilting modules of $G$. The following is a summary of results from \cite[\S3]{CEN}:

\begin{prop}\label{PropVerpG}
Let $\charr\bk=p\geq5$. If $G$ is a simple algebraic group over $\bk$ with $p>h$, then $\fg=\Lie_p(G)$ is an invariantless based Lie algebra in $\Ver_p^+$ which is simple for $p\geq h+2$ and zero for $p=h+1$. We have
\[\Ver_p(G)\simeq\Rep_{\Ver_p^+}(\fg,\phi)\boxtimes\Rep_{\sVec}(Z,z)\simeq\Rep_{\Ver_p}((Z,\fg),(z,\phi))\]
where $Z$ is the centre of $G$, the Harish-Chandra pair $(Z,\fg)$ has $Z$ acting trivially on $\fg$, and $\phi,z$ come from a principal morphism $\SL_2\to G$.
\end{prop}

\begin{proof}
The only part not stated in \cite{CEN} is the right side equivalence. $(Z,\fg)$ is the product of $(1,\fg)$ and $(Z,0)$, and this induces a product in the category of $\Ver_p$-groups. So by \cite[Theorem~4.2.1]{incompressible}, $\Rep((Z,\fg),(z,\phi))$ is the coproduct of $\Rep((1,\fg),(1,\phi))\simeq\Rep_{\Ver_p^+}(\fg,\phi)$ and $\Rep((Z,0),(z,0))\simeq\Rep_{\sVec}(Z,z)$, and the Deligne tensor product $\boxtimes$ is a coproduct (see \cite[\S 2.2.6]{incompressible}).
\end{proof}

\subsection{Orthogonal and symplectic Lie algebras}\label{SecSOSp} Let $f$ be a non-degenerate bilinear form on an object $X\in\cD$, meaning $f$ is a morphism $X\otimes X\to\unit$ such that the composition
\[f':X\xrightarrow{1\otimes\coev_X}X\otimes X\otimes X^*\xrightarrow{f\otimes 1}X^*\]
is an isomorphism. Then $X\otimes X$ is a Lie algebra with bracket $(1\otimes f\otimes 1)(1-\braid_{X^{\otimes2},X^{\otimes2}})$, and $1\otimes f'$ is a Lie algebra isomorphism between $X\otimes X$ and $\fgl(X)=X\otimes X^*$. If $f$ is symmetric, meaning $f=f\circ\braid_{X,X}$, then the dual exterior square $\Lambda^2(X)\subseteq X\otimes X$ is closed under the Lie bracket and thus corresponds to a Lie subalgebra of $\fgl(X)=X\otimes X^*$, which we denote $\fso(X)$. Similarly if $f$ is skew-symmetric, meaning $f=-f\circ\braid_{X,X}$, then the divided square $\Gamma^2(X)\subseteq X\otimes X$ corresponds to a Lie subalgebra of $\fgl(X)$, which we denote $\fsp(X)$.

For $X\in\Ver_p$, we have $\fgl(X)\cong\fgl(X\otimes\bar\unit)$ and similarly for $\fsl$. Skew-symmetric bilinear forms on $X$ are in correspondence with symmetric bilinear forms on $X\otimes\bar\unit$, and thus $\fsp(X)\cong\fso(X\otimes\bar\unit)$. If $V$ is a preimage of $X$ in $\Tilt\SL_2$ and $\fsl(V)$ is also a tilting module, then $\fsl(X)$ is exactly the image of $\fsl(V)$, and similarly for $\fso$ and $\fsp$. Thus, by using \cite[Table~3.2.4]{CEN} to find the image of the natural representation of each classical-type Lie algebra in $\Ver_p$, we obtain the following:
\[\Lie_p(\SL_n)\cong\fsl(L_{n-1})\cong\fsl(L_{p-n-1})\cong\Lie_p(\SL_{p-n})\text{ for }p\geq n+1,\]
\[\Lie_p(\SO_{2n+1})\cong\fso(L_{2n})\cong\fsp(L_{p-2n-2})\cong\Lie_p(\Sp_{p-2n-1})\text{ for }p\geq2n+3,\]
\[\Lie_p(\SO_{2n})\cong\fso(L_0\oplus L_{2n-2})\text{ for }p\geq2n+1.\]

In \cite[\S4]{CEN}, subalgebras of $\fsl(L_{n-1})$ were classified. We now extend this to other Lie algebras of the form $\Lie_p(G)$. First, we give a more explicit description of the Lie bracket on $\fgl(L_n)$, and utilise this to obtain results in type $D$ for arbitrary rank.

\begin{lemma}~\label{LemglImage}
Let $n$ be a positive integer, and identify $\fgl(V_n)$ with $V_n^{\otimes2}$ using the bilinear form $\proj^{n,n}_0$ as in Section~\ref{SecSOSp}. If $\charr\bk=0$, then the $\incl$ and $\proj$ morphisms from Section~\ref{SecInclProj} give a decomposition $\fgl(V_n)=V_n\otimes V_n^*=V_0\oplus V_2\oplus\cdots\oplus V_{2n}$ such that the Lie bracket restricted to $V_a\otimes V_b\to V_c$ is
\[2(-1)^{n+1}\frac{c!D(\frac a2,\frac n2,\frac n2)D(\frac b2,\frac n2,\frac n2)}{a!b!(n+1)!^2D(\frac c2,\frac n2,\frac n2)}\alpha^{a,b,c}_{n,n,n}\proj^{a,b}_c\]
if $\frac{a+b-c}2$ is odd and zero otherwise. If instead $\charr\bk=p\geq n+2$, then $\fgl(V_n)$ has image $\fgl(L_n)\cong L_0\oplus L_2\oplus\cdots\oplus L_{2\min(n,p-n-2)}$ in $\Ver_p$. Then the Lie bracket restricted to the summands $L_a\otimes L_b\to L_c$ is given by the formula above whenever $(a,b,c)$ satisfies the $p$-triangle condition and $\frac{a+b-c}2$ is odd, and zero otherwise.
\end{lemma}

\begin{proof}
The matrix multiplication map $(V_n\otimes V_n^*)\otimes(V_n\otimes V_n^*)\to V_n\otimes V_n^*$ is given by evaluation on the middle components. The following commutative diagram describes the restriction of this map to summands of $V_n\otimes V_n\cong V_n\otimes V_n^*$, where vertical arrows are $\proj$ morphisms:
\[\begin{tikzcd}
(V_n\otimes V_n)\otimes(V_n\otimes V_n) \arrow{d} \arrow{r}{\sim} & V_n\otimes(V_n\otimes(V_n\otimes V_n)) \arrow{d} \arrow{r}{\sim} & V_n\otimes((V_n\otimes V_n)\otimes V_n) \arrow{d}\\
(V_n\otimes V_n)\otimes V_b \arrow{d} \arrow{r}{\sim} & V_n\otimes (V_n\otimes V_b) \arrow{d} & V_n\otimes (V_0\otimes V_n) \arrow{d}{\sim}\\
V_a\otimes V_b \arrow{d} & V_n\otimes V_n \arrow{d} \arrow{r}{\alpha^{n,n,b}_{n,n,0}} & V_n\otimes V_n \arrow{d}\\
V_c \arrow{r}{\alpha^{n,n,a}_{b,c,n}} & V_c \arrow{r} & V_c
\end{tikzcd}\]
We have used Equation~(\ref{EqnRightLeft}) for the second association. For even $a,b,c$ we have
\begin{align*}
\alpha^{n,n,a}_{b,c,n}=\alpha^{b,c,a}_{n,n,n}&=(-1)^{\frac{a-c}2}\frac{c!D(\frac a2,\frac n2,\frac n2)}{a!D(\frac c2,\frac n2,\frac n2)}\alpha^{a,b,c}_{n,n,n},&\alpha^{n,n,b}_{n,n,0}&=(-1)^{n+\frac b2}\frac{D(\frac b2,\frac n2,\frac n2)}{b!(n+1)!^2}
\end{align*}
by Equations~(\ref{EqnAssociator}), (\ref{Eqn6j}) and (\ref{EqnPerms}). The braiding on $V_a\otimes V_b$ restricted to $V_c$ is $(-1)^{(a+b-c)/2}$ by Equation~(\ref{EqnBraidComposition}), so the commutator bracket restricted to $V_a\otimes V_b\to V_c$ is
\[\left(\alpha^{n,n,b}_{n,n,0}\alpha^{b,c,a}_{n,n,n}-(-1)^{\frac{a+b-c}2}\alpha^{n,n,a}_{n,n,0}\alpha^{a,c,b}_{n,n,n}\right)\proj^{a,b}_c\]
which rearranges to the result in the lemma. If $\charr\bk=p$, then for $a\in\{0,2,\dots,2n\}$, $(n,n,a)$ satisfies the $p$-triangle condition if and only if $2n+a\leq 2p-4$, so the results hold for $a\leq 2(p-n-2)$.
\end{proof}

\begin{lemma}\label{LemTypeD}
Let $n$ be an even positive integer, so that $\proj^{n,n}_0\oplus\proj^{0,0}_0$ is a symmetric non-degenerate bilinear form on $V_n\oplus V_0$. We have a decomposition $\fso(V_n\oplus V_0)\cong V_2\oplus V_6\oplus\cdots\oplus V_{2n-2}\oplus W$ with $W\cong V_n$, such that the Lie bracket is:
\begin{align*}
\text{given by Lemma~\ref{LemglImage} on }&V_a\otimes V_b\to V_c, & \alpha^{n,n,a}_{n,n,0}\proj^{a,n}_n\text{ on }&V_a\otimes W\to W,\\
-2\proj^{n,n}_a\text{ on }&W\otimes W\to V_a, & \alpha^{n,n,a}_{n,n,0}\proj^{n,a}_n\text{ on }&W\otimes V_a\to W,
\end{align*}
and zero on all other combinations of summands. The formulas above also hold for $\fso(L_n\oplus\unit)$ in $\Ver_p$, which has decomposition $L_2\oplus L_6\oplus\cdots\oplus L_{2\min(n-1,p-n-2)}\oplus W$ with $W\cong L_n$.
\end{lemma}

\begin{proof}
We have a natural decomposition $\fgl(V_n\oplus V_0)=\fgl(V_n)\oplus V_n\oplus V_n^*\oplus V_0$ as an $\SL_2$-module. Since the braiding on $V_a\subseteq V_n\otimes V_n$ is $(-1)^{\frac a2}$, $\fso(V_n\oplus V_0)$ consists of the summands $V_a\subseteq\fgl(V_n)$ with $\frac a2$ odd, and also a summand $W\cong V_n$ including into $V_n\oplus V_n^*$ with coefficients $(1,-1)$. Now, for summands $V_a\subseteq\fgl(V_n)$, the restrictions of the Lie bracket on $\fgl(V_n\oplus V_0)$ to $V_a\otimes V_n$ and $V_n^*\otimes V_a$ are as follows:
\begin{align*}
V_a\otimes V_n\hookrightarrow(V_n\otimes V_n^*)\otimes V_n \xrightarrow{\sim} V_n\otimes(V_n^*\otimes V_n) \xrightarrow{1\otimes\ev} V_n,\\
V_n^*\otimes V_a\hookrightarrow V_n^*\otimes (V_n\otimes V_n^*) \xrightarrow{\sim} (V_n^*\otimes V_n)\otimes V_n^* \xrightarrow{1\otimes\ev} V_n^*.
\end{align*}
Identifying $V_n$ with $V_n^*$ as in Lemma~\ref{LemglImage}, these morphisms are $\alpha^{n,n,a}_{n,n,0}\proj^{a,n}_n$ and $\alpha^{n,n,a}_{n,n,0}\proj^{n,a}_n$, and this determines the Lie bracket on $V_a\otimes W$ and $W\otimes V_a$. Meanwhile, the Lie bracket on the summands $V_n,V_n^*\subset\fgl(V_n\oplus V_0)$ is given by
\[V_n\otimes V_n^*\xrightarrow{\sim}\fgl(V_n)\text{ and } V_n^*\otimes V_n\xrightarrow{-\braid}V_n\otimes V_n^*\xrightarrow{\sim}\fgl(V_n),\]
and thus the Lie bracket restricted to $W\otimes W\to V_a$ is $-2\proj^{n,n}_a$ (since $\frac a2$ is odd).
\end{proof}

\begin{prop}\label{PropTypeD}
Let $n\geq4$ be an integer such that $\charr\bk=p\geq2n+1$, and let $\fg=\fso(L_0\oplus L_{2n-2})\cong\Lie_p(\SO_{2n})\in\Ver_p$. If $p=17$ and $n=6$ then $\fg$ has a subalgebra with underlying object $L_2\oplus L_{10}$, and this is the only proper subalgebra of $\fg$ not contained in a subalgebra $\fg'$ coming from a principal pair $\SO_{2n-1}\subset\SO_{2n}$, for all $p$ and $n$.
\end{prop}

\begin{proof}
Use the decomposition $\fg=L_2\oplus L_6\oplus\cdots\oplus L_{2\min(n-1,p-n-2)}\oplus W$ from Lemma~\ref{LemTypeD}. Let $N=L_{2n-2}$ be the subobject of $\fg$ whose inclusion is $\id$ on $W$ and $\lambda\incl^{2n-2,2n-2}_{2n-2}$ on $\fg'$ for a fixed $\lambda\in\bk$. Since $\fg=\fg'\oplus W$, we aim to find all cases where $N$ does not generate $\fg'$ inside $\fg$. If $\lambda=0$ then the Lie bracket restricted to $N\otimes N\to\fg'$ is $-2\proj^{2n-2,2n-2}_a$ on each summand $L_a$, which is always non-zero, meaning $N$ generates $\fg'$. So we can assume $\lambda\neq0$, and in particular $\fg'$ has a summand with weight $2n-2$. This forces $n$ to be even and $(2n-2,2n-2,2n-2)$ to satisfy the $p$-triangle condition, that is $p\geq3n-1$. By Lemma~\ref{LemTypeD} and Proposition~\ref{PropRational}, for $u\in\{1,3,\dots,2n-3\}$ the Lie bracket restricted to $N\otimes N\to L_{2u}$ is $b_u\proj^{2n-2,2n-2}_{2u}$ where
\[b_u=\binom{2u}{u}\frac{(2n-2)_u}{(2n+u-1)_u}f_u(n-1,n-1,n-1)\mu-2,\quad \mu=2(-1)^n\frac{(n-1)!^3(3n-2)!}{(2n-1)!^3}\lambda^2\]
using the falling factorial notation from Proposition~\ref{PropRational}. Meanwhile, the Lie bracket restricted to $N\otimes N\to W$ is $b_*\proj^{2n-2,2n-2}_{2n-2}$ where
\[b_*=2\alpha^{2n-2,2n-2,2n-2}_{2n-2,2n-2,0}\lambda=-2(-1)^n\frac{(2n-1)(n-1)!^3(3n-2)!}{(2n-1)!^3}\lambda\]
Note that the image of $N\otimes N\to(V_{2n-2}\oplus W)$ is exactly $N$ if and only if $b_{n-1}=b_*\lambda$. If this does not hold, then $N$ generates $W$ which generates $\fg'$, so we require $b_{n-1}=b_*\lambda$.

First suppose $n=4$, meaning $\fg$ is the image of $\fso_8=V_2\oplus V_6^2\oplus V_{10}$. We can compute that $b_{n-1}=b_*\lambda$ if and only if $\lambda^2=\frac{140}9$, and if this holds then $b_5=-\frac 83\neq0$ for any $p>8$. Thus we potentially have 3 subalgebras of $\fg$ with objects $L_2\oplus L_6\oplus L_{10}$ (or $L_2\oplus L_6$ if $p=11$) corresponding to $\lambda=0$ and the two square roots of $\frac{140}9$. However, these are expected: $\fso_8$ has a triality automorphism $\phi$ with fixed points $\fg_2\cong V_2\oplus V_{10}$, and this has a combinatorial description which is well-defined for $p>2$ \cite{MW}. Since $\phi$ has order three, $\fso_7\subset\fso_8$ can be conjugated by $\phi$ and $\phi^2$ to obtain 3 distinct subalgebras with $\SL_2$ decompositions $V_2\oplus V_6\oplus V_{10}$, and so these must correspond to the Lie subalgebras found above. All of these come from conjugates of the original principal pair $\SO_{2n-1}\subset\SO_{2n}$.

From now on assume $n\geq6$. Since $\fg'=\fso(L_{2n-2})\subseteq\fsl(L_{2n-2})$, we can utilise the proof of \cite[Proposition~4.2.2]{CEN} to determine which summands generate each other going forward. In particular, each of $L_6$, $L_{10}$, $L_{14}$ generates $\fg'$ (the exceptional case with $p=23$ needs either $2n-1$ or $p-2n+1$ to equal 10, which cannot happen for even $n\geq6$).

If $n=6$, then $b_{n-1}=b_*\lambda$ if and only if $\lambda^2=\frac{8316}{41}$, and if this holds then $b_3=-\frac{34}{41}$. This is only zero modulo $p$ if $p=17$, and in this case the summands $L_{14}\oplus L_{18}$ are absent, so $L_2\oplus L_{10}$ forms a subalgebra as required. If instead $n=8$, then $b_{n-1}=b_*\lambda$ if and only if $\lambda^2=\frac{8700120}{7321}$, and if this holds then $b_3=-\frac{7080}{7321}$ and $b_5=-\frac{17160}{7321}$. The numerators of these share no prime factors above 5, so at least one is non-zero and $N$ generates $\fg'$.

Finally suppose $n\geq10$, so $p\geq3n-1\geq29$. We have
\[b_3=\frac{(2n-1)(7n^2-7n+6)}{4(2n+1)(2n-3)}\mu-2,\text{ so }b_3=0\iff\mu=\frac{8(2n+1)(2n-3)}{(2n-1)(7n^2-7n+6)}.\]
Substituting $\mu$ into $b_5$ and $b_7$ gives rational functions in $n$, which one can directly compute:
\begin{align*}
b_5&=-\frac{45n(n-1)(3n^2-3n+8)}{2(2n+3)(2n-5)(7n^2-7n+6)},\\
b_7&=-\frac{33n(n-1)(41n^4-82n^3-671n^2+712n-1260)}{4(2n+3)(2n+5)(2n-5)(2n-7)(7n^2-7n+6)}.
\end{align*}
We require $7n^2-7n+6\neq0$ so that $\mu$ is well-defined, meaning the only terms in these expressions which could be zero modulo $p$ are
\[h_1(n)=3n^2-3n+8\text{ and }h_2(n)=41n^4-82n^3-671n^2+712n-1260.\]
We have $(2464+123n-123n^2)h_1(n)+9h_2(n)=8372$ whose prime factors are all less than 29, so at least one of $b_5$ or $b_7$ is non-zero and thus $N$ generates $\fg'$.
\end{proof}

\begin{theorem}\label{ThmSubalgs}
Let $\bk$ be algebraically closed with $\charr\bk=p\geq5$, and let $\fg=\Lie_p(G)$ for a simple algebraic group $G$ with Coxeter number $h<p$. Then every proper non-zero Lie subalgebra $\fg'\subset\fg$ is one of the following:
\begin{enumerate}
\item\label{Item1P} $\fg'$ is the image of $\Lie(G')$ for some characteristic $p$ reduction of an ordinary principal pair $G'\subset G$ as in \cite[Theorem~6.4]{EO}.
\item\label{Item1D} $\fg'=L_2\oplus L_{p-3}\cong\fso(L_0\oplus L_{p-3})\cong\Lie_p(\SO_{p-1})$ where $G$ is any of these types:
\begin{itemize}
\item $A_{(p-3)/2}$, $A_{(p-1)/2}$ for $p\geq11$,
\item $B_{(p-1)/4}$, $C_{(p-1)/4}$, $D_{(p+3)/4}$ for $p\geq13$ such that $4$ divides $p-1$,
\item $E_6$ for $p\in\{17,19\}$, $E_7$ for $p\in\{29,37\}$, $E_8$ for $p\in\{41,61\}$.
\end{itemize}
\item\label{Item1G} $p=17$, $G$ is type $D_6$, and $\fg'=L_2\oplus L_{10}\cong\Lie_{17}(G_2)$.
\item\label{Item1E7} $p=23$, $G$ is one of the types $A_9,A_{12},B_6,C_5$ or $D_7$, and $\fg'=L_2\oplus L_{14}\cong\Lie_{23}(E_7)$.
\item\label{Item1E8} $p=29$, $G$ is type $E_7$, and $\fg'=L_2\oplus L_{14}\oplus L_{26}\cong\Lie_{29}(E_8,a_1)$.
\end{enumerate}
\end{theorem}

\begin{proof}
Using $\Lie_p(\SL_n)\cong\fsl(L_{n-1})$ for $n<p/2$ and $\fsl(L_{p-n-1})$ for $n>p/2$, \cite[Proposition~4.2.2]{CEN} classifies subalgebras in type $A$; in particular, the subalgebras (b), (c), (d) in that proposition fall under case (\ref{Item1P}), while (e) and (f) fall under (\ref{Item1D}) and (\ref{Item1E7}). Moreover, (b) is exactly the subalgebra $\fso(L_{n-1})$ or $\fsp(L_{n-1})$ as in Section~\ref{SecSOSp} depending on whether $n$ is odd or even, so this also covers types $B$ and $C$. Proposition~\ref{PropTypeD} extends this to type $D$ giving case (\ref{Item1G}). Now, \cite[Theorem~6.4]{EO} states that $F_4$ is contained in $E_6$, while $G_2$ is contained in $A_6$ for $p>7$ and $\Lie_p(G_2)$ is trivial for $p=7$, so classifying subalgebras in type $E$ will complete the theorem. Firstly, the summands of $\Lie_p(E_n)$ for each $p$ and $n$ can be determined using the algorithm in \cite[\S3.2.3]{CEN}, and we state these (excluding the trivial $p=h+1$ case) along with all of the proper non-trivial subalgebras described in the theorem:

\vspace{0.5em}{\noindent\footnotesize
\begin{minipage}{0.5\textwidth}\centering
\begin{tabular}{r|l|l}
$p$ & $\Lie_p(E_6)$ summands & subalgebras \\\hline
$17$ & $2,8,14$ & $(2),(2,14)$\\
$19$ & $2,8,10,16$ & $(2),(2,10),(2,16)$\\
$23$ & $2,8,10,14,16$ & $(2),(2,10,14)$\\
$\geq29$ & $2,8,10,14,16,22$ & $(2),(2,10,14,22)$
\end{tabular}\vspace{0.5em}\\
\begin{tabular}{r|l|l}
$p$ & $\Lie_p(E_7)$ summands & subalgebras \\\hline
$23$ & $2,14$ & $(2)$\\
$29$ & $2,10,14,18,26$ & $(2),(2,26),(2,14,26)$\\
$31$ & $2,10,14,18,22$ & $(2)$\\
$37$ & $2,10,14,18,22,26,34$ & $(2),(2,34)$\\
$\geq41$ & $2,10,14,18,22,26,34$ & $(2)$
\end{tabular}
\end{minipage}\hfill
\begin{minipage}{0.49\textwidth}\centering
\begin{tabular}{r|l|l}
$p$ & $\Lie_p(E_8)$ summands & subalgebras \\\hline
$37$ & $2,22$ & $(2)$\\
$41$ & $2,14,26,38$ & $(2),(2,38)$\\
$43$ & $2,14,22,34$ & $(2)$\\
$47$ & $2,14,22,26,38$ & $(2)$\\
$53$ & $2,14,22,26,34,38$ & $(2)$\\
$59$ & $2,14,22,26,34,38,46$ & $(2)$\\
$61$ & $2,14,22,26,34,38,46,58$ & $(2),(2,58)$\\
$\geq67$ & $2,14,22,26,34,38,46,58$ & $(2)$
\end{tabular}
\end{minipage}}\vspace{0.5em}

To describe the restriction of the Lie bracket to summands $L_a\otimes L_b\to L_c$, we report the results of a SageMath program performing the following computations. By finding highest weight vectors of a principal $\fsl_2$-triple and applying lowering operators, one can obtain a basis of $\Lie(G)$ from which projection and inclusion maps for the $\SL_2$-module summands $V_a\subseteq\Lie(G)$ can be computed over $\mQ$. It can then be verified that these have well-defined characteristic $p$ reductions whenever $L_a$ is a summand of $\Lie_p(G)$. For summands $L_a,L_b,L_c$ of $\fg$, one can compute the Lie bracket on $\fg$ restricted to $L_a\otimes L_b\to L_c$ by first computing the Lie bracket on $\Lie(G)$ over $\mQ$ restricted to $V_a\otimes V_b\to V_c$ in the form $\beta\proj^{a,b}_c$, and then reducing $\beta$ modulo $p$. For $E_7$ with $p=29$, direct computation shows that this restriction is zero for $a,b\in\{2,14,26\}$ and $c\not\in\{2,14,26\}$, and hence case (\ref{Item1E8}) is a true subalgebra. Lastly, we give a selection of bracket restrictions $L_a\otimes L_b\to L_c$ which are non-zero in $\Lie(G)$ in characteristic zero, along with the ``bad primes'' $p$ for which their restrictions are zero in $\fg$ (either because $\beta=0$ modulo $p$ or because one of $L_a,L_b,L_c$ is not a summand of $\fg$). One can quickly see, using these morphisms, that each simple summand generates the minimal subalgebra containing it in every case, so no other subalgebras are possible.

\vspace{0.5em}{\footnotesize\centering\noindent
\begin{tabular}{l|l}
$E_6$ bracket & bad primes\\\hline
$14\otimes14\to2$ & $19$\\
$8\otimes8\to14$ & $19$\\
$10\otimes10\to2$ & $17$\\
$16\otimes16\to2$ & $17$\\
$8\otimes8\to10$ & $17$\\
$8\otimes10\to16$ & $17$\\
$10\otimes10\to14$ & $17,19$\\
$14\otimes14\to10$ & $17,19$\\
$16\otimes16\to10$ & $17,19$\\
$10\otimes16\to8$ & $17$\\
$14\otimes14\to22$ & $17,19,23$\\
$22\otimes22\to10$ & $17,19,23$
\end{tabular}\hfill
\begin{tabular}{l|l}
$E_7$ bracket & bad primes\\\hline
$14\otimes14\to2$ & none\\
$26\otimes26\to2$ & $23,31$\\
$14\otimes14\to26$ & $23,31$\\
$10\otimes10\to14$ & $23$\\
$10\otimes10\to18$ & $23$\\
$18\otimes18\to10$ & $23,71$\\
$26\otimes26\to10$ & $23,29,31$\\
$14\otimes14\to22$ & $23,29$\\
$22\otimes22\to10$ & $23,29,83$\\
$22\otimes22\to18$ & $23,29,31$\\
$34\otimes34\to2$ & $23,29,31$\\
$18\otimes18\to34$ & $23,29,31$\\
$34\otimes34\to10$ & $23,29,31,37$
\end{tabular}\hfill
\begin{tabular}{l|l}
$E_8$ bracket & bad primes\\\hline
$22\otimes22\to2$ & $41$\\
$38\otimes38\to2$ & $37,43$\\
$14\otimes14\to26$ & $37,43$\\
$26\otimes26\to14$ & $37,43$\\
$14\otimes26\to38$ & $37,43$\\
$14\otimes14\to22$ & $37,41$\\
$22\otimes22\to14$ & $37,41$\\
$14\otimes22\to34$ & $37,41,47$\\
$34\otimes34\to14$ & $37,41,47,71$\\
$38\otimes38\to14$ & $37,41,43$\\
$34\otimes34\to38$ & $37,41,43,47,53$\\
$26\otimes26\to46$ & $37,41,43,47,53$\\
$46\otimes46\to14$ & $37,41,43,47,53$\\
$58\otimes58\to2$ & $37,41,43,47,53,59$\\
$22\otimes38\to58$ & $37,41,43,47,53,59$\\
$58\otimes58\to14$ & $37,41,43,47,53,59,61$
\end{tabular}
}\vspace{0.5em}

The isomorphisms in cases (\ref{Item1D}), (\ref{Item1G}), (\ref{Item1E7}), (\ref{Item1E8}) may be checked with direct computations. However, for the sake of brevity, we will simply invoke Theorems~\ref{ThmLength2} and \ref{ThmLength3} from later in the paper (note that in each case $\fg'$ has no proper non-trivial subalgebras and contains $L_2$, and is thus simple and based). These do not rely on Theorem~\ref{ThmSubalgs}, so this is not circular.
\end{proof}

\begin{remark}\label{RemE8a1}
In the proof of Theorem~\ref{ThmLength3}, we show that $L_2\oplus L_{14}\oplus L_{26}\in\Ver_{29}^+$ has a unique simple based Lie algebra structure, and so the exceptional subalgebra $\fg'$ of $\Lie_p(E_7)$ in Theorem~\ref{ThmSubalgs} must be isomorphic to this Lie algebra. One can observe from \cite[Table~10]{St} that $\fg'\cong\Lie_{29}(E_8,a_1)$, where $a_1$ corresponds to the nilpotent in $\Lie(G)$ given in \cite[Table~8]{St}. It is not clear whether $\Rep_{\Ver_{29}^+}(\fg',\phi)$ is equivalent to the semisimplification of $\Tilt E_8$, in a similar manner to the results of \cite{BEEO} for $\GL_n$. However, since $\fg'$ appears as a subalgebra of the Lie algebra $\fg$ coming from $E_7$ by Theorem~\ref{ThmSubalgs}, we do know that $\fg'$ is \textbf{linearly reductive}, meaning $\Rep(\fg',\phi)$ is semisimple, by the following argument. By Corollary~\ref{CorBased}, there are affine group schemes $G,G'$ over $\Ver_{29}$ such that $\Rep(\fg,\phi)\simeq\Rep(G,\phi)$ and $\Rep(\fg',\phi)\simeq\Rep(G',\phi)$ for appropriate constraints $\phi$. The monomorphism $\fg'\hookrightarrow\fg$ corresponds to a monomorphism $G'\hookrightarrow G$, and thus the restriction functor $F:\Rep(\fg,\phi)\to\Rep(\fg',\phi)$ is surjective in the sense of \cite{incompressible}. We know $\Rep(\fg,\phi)\simeq\Ver_{29}(E_7)$ is semisimple, and thus $\Rep(\fg',\phi)$ is semisimple by \cite[Example~3.3(ii),(iii)]{ftc}.
\end{remark}

\begin{remark}
Theorem~\ref{ThmSubalgs} gives a complete answer to \cite[Question~6.5]{EO}, asking about surjective tensor functors out of the category $\Ver_p(G)$ for simple algebraic groups $G$. By Proposition~\ref{PropVerpG}, surjective tensor functors out of $\Ver_p(G)$ correspond to subgroups of the affine group scheme corresponding to the Harish-Chandra pair $(Z,\fg)$ where $Z$ is the centre of $G$. This is equivalent to pairs $(Z',\fg')$ with $Z'\subseteq Z$ a subgroup and $\fg'\subseteq\fg$ a Lie subalgebra, and the latter is characterised by Theorem~\ref{ThmSubalgs} (and \cite[Proposition~4.3.2]{CEN} for isomorphisms). Up to autoequivalences and for large enough $p$, the only negative answer to \cite[Question~6.5]{EO} is given by Theorem~\ref{ThmSubalgs}(\ref{Item1D}).
\end{remark}

\subsection{Lie superalgebras and exceptional cases}\label{SecSuper} A \textbf{Lie superalgebra} in a symmetric tensor category $\cD$ is a Lie algebra in $\cD\boxtimes\sVec$. If there is a canonical inclusion $F:\sVec\hookrightarrow\cD$, for instance if $\cD=\Ver_p$, then any Lie superalgebra $X\boxtimes\unit\oplus Y\boxtimes\bar\unit$ can be converted into a Lie algebra structure on $X\oplus Y\otimes F(\bar\unit)$. This gives another method for producing Lie algebras in $\Ver_p$; for example, recall from \cite[Remark~4.3.3]{CEN} that the Lie algebra $\Lie_p(\SO_{2n})$ can also be interpreted as an image of $\fosp(1|p-2n+1)$. Non-principal restrictions can give other isomorphisms, see for example cases (\ref{Item4B}) and (\ref{Item4D}) in Section~\ref{SecComp}.

All of the Lie algebras appearing in Theorem~\ref{ThmMain} arise from ordinary simple algebraic groups. However, in Section~\ref{SecComp} we state the results of a computational search for length 4 simple based Lie algebras, and find two instances which require Lie superalgebras instead. Proposition~\ref{PropHQ} shows that these give a negative answer to \cite[Question~4.6]{asymptotic}, asking whether all invariantless simple Lie algebras in $\Ver_p$ come from ordinary Lie algebras. First, let us describe the constructions of these Lie algebras:

\begin{enumerate}
\item The strange Lie superalgebra $\mathfrak{psq}_3$ (Kac type $Q(2)$) has even part $\fsl_3$ and odd part the adjoint representation of $\fsl_3$. Taking $\phi:\fsl_2\to\mathfrak{psq}_3$ to be a principal map into the even part, we obtain a Lie superalgebra in $\Ver_p$ with even part $L_2\oplus L_4$ and odd part $L_2\oplus L_4$. Tensoring the odd part by $\bar\unit$ gives a simple based Lie algebra structure on $L_2\oplus L_4\oplus L_{p-6}\oplus L_{p-4}$ for $p\geq7$.
\item For a positive integer $n$, write $W(n)$ for the superalgebra of derivations of $\Lambda(\bk^n)$, and write $\widetilde{H}(n)$ for the image of the map $\Lambda(\bk^n)\to W(n)$ given by $f\mapsto\{f,-\}$ where $\{,\}$ is induced by the standard nondegenerate symmetric bilinear form on $\bk^n$ (see \cite[\S 3.3.2]{Kac}). This map has kernel $\Lambda^0(\bk^n)$. The Cartan type Hamiltonian Lie superalgebra $H(n)$ is the derived subalgebra of $\widetilde H(n)$, or equivalently the image of $\bigoplus_{i=1}^{n-1}\Lambda^i(\bk^n)$. Now, the restriction of the natural 5-dimensional representation of $\SO_5$ to $\SL_2$ via a principal morphism is $V_4$, giving an inclusion $\fso(V_4)\cong\Lambda^2(\bk^5)\hookrightarrow H(5)$. This turns $H(5)$ into a Lie superalgebra in $\Rep\SL_2$ with underlying object
\[\bigoplus_{i=1}^4\Lambda^i(\bk^5)\boxtimes\bar\unit^{\otimes i}\cong V_4\boxtimes\bar\unit\oplus(V_2\oplus V_6)\boxtimes\unit\oplus(V_2\oplus V_6)\boxtimes\bar\unit\oplus V_4\boxtimes\unit.\]
This is a tilting module, and for $p=7$ its image in $\Ver_p$ is a Lie superalgebra with even part $L_2\oplus L_4$ and odd part $L_2\oplus L_4$, which can then be turned into a simple based Lie algebra structure on $L_1\oplus L_2\oplus L_3\oplus L_4$. Note that this has the same underlying object as the $Q(2)$ image, but it is not isomorphic: since the Poisson bracket has degree $-2$ with respect to the grading on $\Lambda(\bk^n)$, the bracket map $L_4\otimes L_4\to L_2$ is zero, while this map is non-zero for $Q(2)$ in which the even part is $\Lie_p(\SL_3)$.
\end{enumerate}

The \textbf{Killing form} $\kappa$ on a Lie algebra $\fg$ in a symmetric tensor category is the composition
\[\fg\otimes\fg\xrightarrow{1\otimes1\otimes\coev}\fg\otimes\fg\otimes\fg\otimes\fg^*\xrightarrow{[,[,]]\otimes1}\fg\otimes\fg^*\xrightarrow{\ev}\unit,\]
and this is preserved by monoidal functors. In $\sVec$, we have $\kappa(x,y)=\operatorname{str}(\ad(x)\circ\ad(y))$ where $\operatorname{str}$ is the supertrace, defined as the trace of the restriction to even components minus the trace of the restriction to odd components. We shall use this to prove the following proposition:

\begin{prop}\label{PropHQ}
The Lie algebra structures on $L_2\oplus L_4\oplus L_{p-6}\oplus L_{p-4}\in\Ver_p$ arising from $Q(2)$ for $p\geq7$ and $H(5)$ for $p=7$ have zero Killing form, and are not isomorphic to $\Lie_p(G,\phi)$ for any $G$ or $\phi$.
\end{prop}

\begin{proof}
The Killing form is zero on the type $Q$ Lie superalgebras (see \cite[\S2.4]{Kac}), and it is also zero on $H(5)$ as follows. Since the Poisson bracket on $\Lambda(\bk^5)$ has degree $-2$, $H(5)$ has a $\mZ$-grading with components $H_i=\Lambda^{i+2}(\bk^5)$ for $-1\leq i\leq 2$. The grading means that the Killing form is only non-zero on $H_j\otimes H_{-j}$ for $j\in\{-1,0,1\}$. For $x,y\in H_0$, the perfect pairing of $\Lambda^i(\bk^5)$ with $\Lambda^{5-i}(\bk^5)$ means that $\ad(x)\circ\ad(y)$ has the same trace on $H_i$ and $H_{1-i}$ but with opposite parity, and thus $\kappa$ is zero on $H_0\otimes H_0$. Meanwhile, the Killing form on $H_1\otimes H_{-1}$ corresponds to a $H_0$-module map $H_1\to H_{-1}^*$, but restricting along a principal morphism $\fsl_2\hookrightarrow\fso_5\cong H_0$ turns this into a map $\Lambda^3(V_4)\cong V_2\oplus V_6\to\Lambda^1(V_4)^*\cong V_4$ which must be zero. Thus, $\kappa$ is zero on $H(5)$ and its image in $\Ver_p$.

Suppose there exist $G$, $\phi$ and $p\geq7$ such that $\fg=\Lie_p(G,\phi)$ has $\kappa=0$ and underlying object $L_2\oplus L_4\oplus L_{p-6}\oplus L_{p-4}$. Since negligible tilting modules in $\Tilt\SL_2$ have dimension a multiple of $p$, the dimension of $\Lie(G)$ equals the categorical dimension of $\fg$ modulo $p$, which is $3+5-5-3=0$. This shows that $p$ does not divide the dual Coxeter number $h^\vee$:
\begin{enumerate}
\item In type $A_n$, if $p$ divides $h^\vee=n+1$ then $\dim(\Lie(G))=(n+1)^2-1=-1\mod p$.
\item In type $B_n$, if $p$ divides $h^\vee=2n-1$ then $\dim(\Lie(G))=n(2n+1)=1\mod p$.
\item In type $C_n$, if $p$ divides $h^\vee=n+1$ then $\dim(\Lie(G))=n(2n+1)=1\mod p$.
\item In type $D_n$, if $p$ divides $h^\vee=2(n-1)$ then $\dim(\Lie(G))=n(2n-1)=1\mod p$.
\item In the exceptional types, $h^\vee$ has no prime divisors above 5.
\end{enumerate}
Now, the form $\kappa/(2h^\vee)$ on $\Lie(G)$ is integral, and the determinant of $\kappa/(2h^\vee)$ equals $h^\vee$ if $G$ has type $A$ and otherwise has no prime factors above 5 (see the definitions and table in \cite[\S 5]{GN}). This means that for $p\geq7$ not dividing $h^\vee$, $\kappa$ is non-degenerate on $\Lie(G)$ and thus corresponds to an isomorphism $\Lie(G)\to\Lie(G)^*$. This descends to an isomorphism $\fg\to\fg^*$ in $\Ver_p$ corresponding to $\kappa$ on $\fg$, contradicting $\kappa=0$.
\end{proof}

\section{Combinatorics of operadic Lie algebras in $\Ver_p$}\label{SecLieForm}

We will now rewrite the axioms of operadic Lie algebras in $\Ver_p$ in terms of the decompositions from Section~\ref{SecPrelim}. This will be used to prove part (a) of Theorem~\ref{ThmMain}.

Let $\fg$ be an object in $\Ver_p$ with $p\geq5$, and fix a decomposition into simples $\fg=\bigoplus_i M_i$ with $M_i=L_{n_i}$ for some $n_i\in\{0,\dots,p-2\}$. Let $[,]:\fg\otimes\fg\to\fg$ be a morphism, and write $\beta^{i,j}_k$ for the scalar such that $\beta^{i,j}_k\proj_{n_k}^{n_i,n_j}$ is equal to the morphism
\[M_i\otimes M_j\hookrightarrow\fg\otimes\fg\xrightarrow{[,]}\fg\twoheadrightarrow M_k,\]
with $\beta^{i,j}_k=0$ whenever $(n_i,n_j,n_k)$ does not satisfy the $p$-triangle condition.

\begin{prop}\label{PropLieForm}
$\fg$ is an operadic Lie algebra with bracket $[,]$ if and only if
\begin{equation}\label{EqnAntisym}
\beta^{i,j}_k=(-1)^{1+(n_i+n_j-n_k)/2}\beta^{j,i}_k
\end{equation}
for all indices $i,j,k$ such that $n_i+n_j-n_k$ is even (antisymmetry), and
\begin{equation}\label{EqnJacobi}
\begin{split}
\left(\sum_{m\text{ with }n_m=r}\beta^{m,k}_l\beta^{i,j}_m\right)&-\left(\sum_m\beta^{i,m}_l\beta^{j,k}_m\alpha^{n_i,n_j,r}_{n_k,n_l,n_m}\right)\\
+&\left((-1)^{(n_i+n_j-r)/2}\sum_m\beta^{j,m}_l\beta^{i,k}_m\alpha^{n_j,n_i,r}_{n_k,n_l,n_m}\right)=0
\end{split}
\end{equation}
for all indices $i,j,k,l$ and $r\in\{0,\dots,p-2\}$ such that $n_i+n_j-r$ is even (Jacobi identity). If additionally $M_0=L_2$ and $(\fg,\phi)$ is a based Lie algebra where $\phi$ is the inclusion of $M_0$, then
\[\beta^{0,i}_j=\beta^{i,0}_j=\begin{cases}\frac{n_i+2}4&i=j\text{ and }n_i\neq0\\0&i\neq j\text{ or }n_i=0\end{cases}\]
for all indices $i,j$. 
\end{prop}

\begin{proof}
Equation~(\ref{EqnAntisym}) follows from Equation~(\ref{EqnBraidComposition}). The Jacobi identity requires that the morphism $\Sigma:\fg^{\otimes3}\to\fg$ is zero, where $\Sigma$ is the sum of the following morphisms:
\begin{align*}
&(\fg\otimes\fg)\otimes\fg\xrightarrow{[,]\otimes1}\fg\otimes\fg\xrightarrow{[,]}\fg\\
&(\fg\otimes\fg)\otimes\fg\xrightarrow{\text{assoc.}}\fg\otimes(\fg\otimes\fg)\xrightarrow{1\otimes[,]}\fg\otimes\fg\xrightarrow{-[,]}\fg\\
&(\fg\otimes\fg)\otimes\fg\xrightarrow{\text{braid}\otimes1}(\fg\otimes\fg)\otimes\fg\xrightarrow{\text{assoc.}}\fg\otimes(\fg\otimes\fg)\xrightarrow{-1\otimes[,]}\fg\otimes\fg\xrightarrow{-[,]}\fg
\end{align*}
Splitting this over summands of $\fg$, the Jacobi identity is equivalent to the morphism
\[L_s\xrightarrow{\incl_s^{r,n_k}} L_r\otimes M_k\xrightarrow{\incl_r^{n_i,n_j}\otimes1}(M_i\otimes M_j)\otimes M_k\hookrightarrow\fg^{\otimes3}\xrightarrow{\Sigma}\fg\twoheadrightarrow M_l\]
being zero for all indices $i,j,k,l$ and $r,s\in\{0,\dots,p-2\}$. This is automatically zero if $s\neq n_l$, and otherwise it is the identity morphism on $L_{n_l}$ times the scalar quantity in Equation~(\ref{EqnJacobi}). Finally, the action of $\fsl(L_1)$ on a based Lie algebra follows from Section~\ref{SecSL2Lie}.
\end{proof}

\begin{remark}
This process also gives formulas for representations of $\fg$. Suppose $X\in\Ver_p$ has decomposition $X=\bigoplus_i X_i$, $X_i=L_{x_i}$, and we have morphisms
\[\rho_k^{i,j}\proj_{x_k}^{n_i,x_j}:M_i\otimes X_j\to X_k.\]
The condition for $(X,\rho)$ to be a representation of $\fg$ is ``the Jacobi identity on $\fg\otimes\fg\otimes X$.'' That is, we take the 3 morphisms on $\fg\otimes\fg\otimes\fg$ in the proof of Proposition~\ref{PropLieForm}, replace the rightmost $\fg$ with $X$ and replace brackets with $\rho$ where necessary, and then require that the sum equal zero. This means $(X,\rho)$ is a $\fg$-module if and only if Equation~(\ref{EqnJacobi}) holds after replacing appropriate $n$'s by $x$'s and $\beta$'s by $\rho$'s.
\end{remark}

The following two lemmas will be useful for proving Theorem~\ref{ThmMain}.

\begin{lemma}\label{LemGoodDecomp}
Suppose $\fg$ is a based Lie algebra in $\Ver_p$ with $p\geq5$, and $\fg\not\in\sVec\subset\Ver_p$. Then there exists a decomposition $\fg=\bigoplus_i M_i$, with $M_0$ the canonical $\fsl(L_1)$, for which $\beta^{i,j}_0=0$ whenever $i,j$ are such that $i\neq j$, $n_i=n_j$ and $n_i,n_j$ are even.
\end{lemma}

\begin{proof}
Fix a decomposition $\fg=\bigoplus_i M_i$, and suppose $i_1,\dots,i_m$ are distinct non-zero indices such that $M_{i_1}=\cdots=M_{i_m}=L_n$. If $n$ is even, then by Equation~(\ref{EqnAntisym}) the restriction of $\beta$ to $(\bigoplus_j M_{i_j})^{\otimes2}\to M_0$ corresponds to a symmetric bilinear form $\bk^m\times\bk^m\to\bk$. By \cite[Theorem~6.5]{Jac} there exists a basis of $\bk^m$ for which this bilinear form is diagonal, meaning we can choose a new decomposition $\bigoplus_j M'_{i_j}$ of $\bigoplus_j M_{i_j}$ for which $\beta^{i_j,i_k}_0=0$ for all $j\neq k$. Repeating for the isomorphism class of $L_n$ for each even $n$ (and excluding $M_0$ when considering the isomorphism class of $L_2$) gives the required decomposition.
\end{proof}

\begin{lemma}\label{LemL2Alg}
Suppose $p\geq5$ and $\fg$ is a Lie algebra in $\Ver_p$ with underlying object $L_2^{\oplus n}=A\otimes L_2$ where $A$ is an $n$-dimensional vector space. Then the Lie bracket is $\nabla\otimes\proj^{2,2}_2$ for some operation $\nabla:A\otimes A\to A$ turning $A$ into a (possibly non-unital) commutative associative $\bk$-algebra. If additionally $\fg$ is a based Lie algebra, then $A$ is also unital.
\end{lemma}

\begin{proof}
Equation~(\ref{EqnAntisym}) becomes $\beta^{i,j}_k=\beta^{j,i}_k$ for all indices $i,j,k$, and Equation~(\ref{EqnJacobi}) becomes
\[\delta_{r=2}\left(\sum_m\beta^{m,k}_l\beta^{i,j}_m\right)+\sum_m\left((-1)^{r/2}\beta^{j,m}_l\beta^{i,k}_m-\beta^{i,m}_l\beta^{j,k}_m\right)\alpha^{2,2,r}_{2,2,2}=0.\]
For all $p$ we have $\alpha^{2,2,0}_{2,2,2}=-1$ for $r=0$ and $\alpha^{2,2,2}_{2,2,2}=1/2$ for $r=2$. Any other value of $r$ gives an equation which is either trivial or equivalent to the equation for $r=0$. Thus, the complete set of Jacobi equations is
\[\sum_m\left(2\beta^{m,k}_l\beta^{i,j}_m-\beta^{j,m}_l\beta^{i,k}_m-\beta^{i,m}_l\beta^{j,k}_m\right)=0,\quad
\sum_m\left(\beta^{j,m}_l\beta^{i,k}_m-\beta^{i,m}_l\beta^{j,k}_m\right)=0\]
for all indices $i,j,k,l$. Note that the first equation is actually a linear combination of two copies of the second equation, one with $i$ and $k$ swapped and the other with $j$ and $k$ swapped. Fix a basis $\{v_i\}$ of $A$, and define $\nabla$ by $\nabla(v_i,v_j)=\sum_m\beta^{i,j}_m v_m$. Then the equations above are exactly the equations which must be satisfied for $\nabla$ to be commutative and associative. In particular, the $v_l$ components of $\nabla(\nabla(v_i,v_k),v_j)$ and $\nabla(v_i,\nabla(v_k,v_j))$ are $\sum_m\beta^{j,m}_l\beta^{i,k}_m$ and $\sum_m\beta^{i,m}_l\beta^{j,k}_m$ respectively. Finally, if $\fg$ is based with $\fsl(L_1)$-subalgebra in index 0, then $\beta^{0,i}_j=\beta^{i,0}_j=\delta_{i=j}$ for all $i,j$, and thus $v_0$ is an identity element.
\end{proof}

\begin{remark}
Lemma~\ref{LemL2Alg} can also be derived from \cite{EKO}, where Lie algebras of the form $\unit^{\oplus m}\oplus L_2^{\oplus n}$ are studied. 
\end{remark}

\begin{theorem}\label{ThmLength1}
Let $p$ be a prime. Every operadic Lie algebra in $\Ver_p$ with length 1 is isomorphic to either $\fsl(L_1)$ (which only exists for $p\geq5$) or the trivial Lie algebra on $L_n$ for some $n\in\{0,\dots,p-2\}$.
\end{theorem}

\begin{proof}
The cases $\Ver_2\simeq\Vecc$ and $\Ver_3\simeq\sVec$ are trivial, so assume $p\geq5$. Suppose $\fg=L_n$ is an operadic Lie algebra, and let the Lie bracket be $\beta\proj^{n,n}_n:L_n\otimes L_n\to L_n$ with $\beta\in\bk$. If two Lie algebras $\fg,\fg'$ with the same underlying object $L_n$ have $\beta\neq0$ and $\beta'\neq0$, then we have an isomorphism $f:\fg\to\fg'$ given by scaling by $\beta/\beta'$, so there is at most one non-trivial Lie algebra up to isomorphism on each $L_n$. We now classify all $(n,p)$ for which a non-trivial Lie algebra $\fg$ exists.

If $n$ is odd then $\beta$ must be zero, so $n$ is even. If $n/2$ is even then antisymmetry gives $\beta=-\beta$ and hence $\beta=0$, so $n/2$ must be odd. The Jacobi identity is
\[\left(\delta_{n=r}+\left((-1)^{r/2}-1\right)\alpha^{n,n,r}_{n,n,n}\right)\beta^2=0\text{ for }0\leq r\leq\min(2n,p-2).\]
We have $\alpha^{2,2,2}_{2,2,2}=1/2$, and hence for $n=r=2$ the coefficient of $\beta^2$ is zero and the equation holds automatically. The coefficient is also zero for $r\in\{0,4\}$, so $n=2$ has a non-trivial Lie algebra, which is of course $\fsl(L_1)$. Now, for even $n\in\mN_{>1}$ we have
\[\alpha^{n,n,2}_{n,n,n}=-\frac{(-1)^{n/2}(n+1)(n+2)n!^3}{4n(\frac n2)!^3(\frac{3n}2+1)!}.\]
Since $n$ is even we have $n\leq p-3$, and $(n,n,n)$ satisfies the $p$-triangle condition since otherwise $\beta=0$, thus $\alpha^{n,n,2}_{n,n,n}\neq0$ modulo $p$. Thus for $n\neq2$ and $r=2$, we get $\beta^2=0$, so there is no non-trivial operadic Lie algebra on $L_n$.
\end{proof}

\begin{theorem}\label{ThmLength2}
Let $\bk$ be algebraically closed with $\charr\bk=p\geq5$, and suppose $\fg$ is a simple based Lie algebra in $\Ver_p$ of length 2. Then $\fg$ is isomorphic to one of the following:
\begin{enumerate}
\item\label{Item2A} $\Lie_p(\SL_3)\cong\fsl(L_2)=L_2\oplus L_4$ for $p\geq7$;
\item\label{Item2B} $\Lie_p(\SO_5)\cong\fso(L_4)=L_2\oplus L_6$ for $p\geq11$;
\item\label{Item2D} $\Lie_p(\SO_{p-1})\cong\fso(L_0\oplus L_{p-3})=L_2\oplus L_{p-3}$ for $p\geq11$;
\item\label{Item2G} $\Lie_p(G_2)=L_2\oplus L_{10}$ for $p\geq17$;
\item\label{Item2E7} $\Lie_{23}(E_7)=L_2\oplus L_{14}$ for $p=23$;
\item\label{Item2E8} $\Lie_{37}(E_8)=L_2\oplus L_{22}$ for $p=37$.
\end{enumerate}
\end{theorem}

Note that these are all the Lie algebras listed in \cite[\S4.3]{asymptotic}, giving an affirmative answer to \cite[Conjecture~4.4]{asymptotic}.

\begin{proof}
Simple based Lie algebra structures on these objects exist by the methods of Section~\ref{SecOrdinary}, so we just need to show that any simple based Lie algebra $\fg$ is isomorphic to one of these. By Lemma~\ref{LemLieSuper} we can assume $\fg\not\in\sVec\subset\Ver_p$, so $\fg=M_0\oplus M_1$ with $M_0=\fsl(L_1)$ acting canonically. Let $M_1=L_n$. The unknown values to determine are $\beta^{1,1}_0$ and $\beta^{1,1}_1$, the Lie bracket on $M_1\otimes M_1$. If $n$ is odd then antisymmetry gives $\beta^{1,1}_0=0$, hence $\fg$ is non-simple with ideal $M_1$. If $n=2$ then by Lemma~\ref{LemL2Alg}, $\fg=A\otimes L_2$ where $A$ is a 2-dimensional commutative algebra. But this is necessarily non-simple, and an ideal $I$ of $A$ gives an ideal $I\otimes L_2$ of $\fg$. Thus $n$ is even and at least 4.

In Equation~(\ref{EqnJacobi}), we take $i=j=k=l=1$ giving
\begin{equation}\label{EqnRank2}
\left(\begin{cases}\beta^{0,1}_1\beta^{1,1}_0&\text{if }r=2\\(\beta^{1,1}_1)^2&\text{if }r=n\\0&\text{otherwise}\end{cases}\right)+\left((-1)^{r/2}-1\right)\left(\beta^{0,1}_1\beta^{1,1}_0\alpha^{n,n,r}_{n,n,2}+(\beta^{1,1}_1)^2\alpha^{n,n,r}_{n,n,n}\right)=0.
\end{equation}
First suppose $\beta^{1,1}_1=0$. Taking $r=2$ gives $\beta^{0,1}_1\beta^{1,1}_0(1-2\alpha^{n,n,2}_{n,n,2})=0$, but we have
\[\alpha^{n,n,2}_{n,n,2}=\frac{3(n-1)!^2}{(n+2)!^2}\sum_{t=n+1}^{n+2}\frac{(-1)^t(t+1)!}{(t-n-1)!^4(2n-t)!(n+2-t)!^2}=\frac{3(n^2+2n-4)}{n(n+1)(n+2)},\]
\[1-2\alpha^{n,n,2}_{n,n,2}=\frac{(n+3)(n-2)(n-4)}{n(n+1)(n+2)}.\]
This can only be zero modulo $p$ if $n=4$ or $n=p-3$, which are cases (\ref{Item2A}) and (\ref{Item2D}). All non-zero values of $\beta^{1,1}_0$ give isomorphic Lie algebras by rescaling $M_1$, so there is a unique Lie algebra in each case.

Now suppose $\beta^{1,1}_1\neq0$, which can only happen if $n/2$ is odd (by considering $r=n$ in Equation~(\ref{EqnRank2})) and $L_n\subseteq L_n\otimes L_n$ (so $3n\leq2p-4$). We now have
\begin{align*}
\beta^{0,1}_1\beta^{1,1}_0(1-2\alpha^{n,n,2}_{n,n,2})-2(\beta^{1,1}_1)^2\alpha^{n,n,2}_{n,n,n}&=0\text{ by taking $r=2$, and }\\
\beta^{0,1}_1\beta^{1,1}_0\alpha^{n,n,r}_{n,n,2}+(\beta^{1,1}_1)^2\alpha^{n,n,r}_{n,n,n}&=0\text{ for $r/2$ odd, $2<r<n$.}
\end{align*}
As shown above, $1-2\alpha^{n,n,2}_{n,n,2}$ in the first equation is non-zero, so the value of $\beta^{1,1}_0$ is determined by the value of $\beta^{1,1}_1$ and the solution space is at most 1-dimensional. Thus all non-trivial solutions give isomorphic Lie algebras via rescaling $M_1$, so there is at most one Lie algebra up to isomorphism for each $n$. Taking $n=6$ and $n=10$ gives cases (\ref{Item2B}) and (\ref{Item2G}) respectively, so suppose $n\geq14$. For each $r\in\{6,10\}$, the two equations above are a linear system in $\beta^{0,1}_1\beta^{1,1}_0=\frac{n+2}4\beta^{1,1}_0$ and $(\beta^{1,1}_1)^2$, which has a non-zero solution when the determinant
\[d_r\coloneqq(1-2\alpha^{n,n,2}_{n,n,2})\alpha^{n,n,r}_{n,n,n}+2\alpha^{n,n,2}_{n,n,n}\alpha^{n,n,r}_{n,n,2}\]
is zero. Using Proposition~\ref{PropRational}, one can compute
\begin{align*}
d_6&=-\frac{(x+2)(x-1)(x-2)(x-5)(7x-3)(2x-3)!^2(2x+4)!}{20(x+1)(2x+1)x!^3(3x+1)!},\\
d_{10}&=\frac{(x+2)(x-3)(23x^4+384x^3+1713x^2-4528x+1680)(2x+6)!(2x-2)!(2x-5)!}{1008(2x+1)(2x+2)(2x+5)(2x-3)x!^3(3x+1)!}
\end{align*}
where $x=n/2$. Since $3n\leq2p-4$ and hence $3x\leq p-2$, the only terms in these expressions which could be zero modulo $p$ are
\[h_1(x)=7x-3\text{ and }h_2(x)=23x^4+384x^3+1713x^2-4528x+1680\]
which are in $d_6$ and $d_{10}$ respectively. We have
\[2401h_2(x)-(7889x^3+135093x^2+645456x-1276480)h_1(x)=204240=2^4\cdot3\cdot5\cdot23\cdot37,\]
so in order to have $h_1(x)=h_2(x)=0$ modulo $p$ we require $p=23$ or $p=37$. The only solution to $h_1(x)=h_2(x)=0$ modulo $23$ is $x=7$, giving $\fg=L_2\oplus L_{14}$ which is case (\ref{Item2E7}). The only solution modulo $37$ is $x=11$, giving $\fg=L_2\oplus L_{22}$ which is case (\ref{Item2E8}).
\end{proof}

\begin{remark}
If $\bk$ is not algebraically closed, then for any 2-dimensional field extension $A$ of $\bk$ we obtain another simple based Lie algebra $A\otimes L_2$. Aside from this, we have the same list of objects but with potentially more isomorphism classes of Lie algebras. In particular, rescaling $M_1$ by $\lambda$ amounts to rescaling $\beta^{1,1}_1$ by $\lambda$ and $\beta^{1,1}_0$ by $\lambda^2$. We have $\beta^{1,1}_1=0$ in cases (\ref{Item2A}) and (\ref{Item2D}), so for these the number of isomorphism classes of Lie algebras is the number of square classes, i.e. the index of $(\bk^\times)^2$ in $\bk^\times$. For the other cases, we observe in the proof that $\beta^{1,1}_0$ is determined by $\beta^{1,1}_1$, so there is exactly one isomorphism class of Lie algebras in each case for all $\bk$.
\end{remark}

\section{Simple based Lie algebras of length 3}\label{SecLength3}

\begin{theorem}\label{ThmLength3}
Let $\bk$ be algebraically closed with $\charr\bk=p\geq5$, and suppose $\fg$ is a simple based Lie algebra in $\Ver_p$ of length 3. Then either $\fg\cong\fsl_2\in\Vecc\subset\Ver_p$, or $\fg$ is isomorphic to one of the following:
\begin{enumerate}
\item\label{Item3A} $\Lie_p(\SL_4)\cong\fsl(L_3)=L_2\oplus L_4\oplus L_6$ for $p\geq11$;
\item\label{Item3B} $\Lie_p(\SO_7)\cong\fso(L_6)=L_2\oplus L_6\oplus L_{10}$ for $p\geq13$;
\item\label{Item3C} $\Lie_p(\Sp_6)\cong\fso(L_{p-7})=L_2\oplus L_6\oplus L_{10}$ for $p\geq17$;
\item\label{Item3D} $\Lie_p(\SO_{p-3})\cong\fso(L_0\oplus L_{p-5})=L_2\oplus L_6\oplus L_{p-5}$ for $p\geq11$;
\item\label{Item3E6} $\Lie_{17}(E_6)=L_2\oplus L_8\oplus L_{14}$;
\item\label{Item3F4} $\Lie_{23}(F_4)=L_2\oplus L_{10}\oplus L_{14}$;
\item\label{Item3E8} $\Lie_{29}(E_8,a_1)=L_2\oplus L_{14}\oplus L_{26}$;
\end{enumerate}
\end{theorem}

The entirety of this section is devoted to proving Theorem~\ref{ThmLength3}. Note that $\Lie_p(\Sp_6)$ is also well-defined for $p=13$, but this is isomorphic to $\Lie_p(\SO_7)$ as observed in \cite[Proposition~4.3.2(3)]{CEN}. For $p\geq17$, these Lie algebras are distinct; the restrictions of the Lie bracket to $L_{10}\otimes L_{10}\to L_6$ and $L_6\otimes L_{10}\to L_{10}$ are zero in case (\ref{Item3B}) and non-zero in case (\ref{Item3C}).

Simple Lie algebra structures on these objects exist by the methods of Section~\ref{SecOrdinary}, and we need to show that any simple based Lie algebra $\fg$ is isomorphic to one of these. If $\fg$ is in the subcategory $\Vecc$ or $\sVec$ then we can apply Lemma~\ref{LemLieSuper}, so we assume from now on that $\fg\not\in\sVec$, meaning we have a decomposition $\fg=M_0\oplus M_1\oplus M_2$ with $M_0=\fsl(L_1)$ acting canonically. Throughout this section, $M_1=L_{n_1}$ and $M_2=L_{n_2}$. We will first show that the objects listed in the theorem are the only objects on which there can exist a simple Lie algebra structure. This is done in multiple steps: first we eliminate cases with $n_1$ or $n_2$ odd, zero or equal to 2 in Sections~\ref{SecOddInv} and \ref{SecMultipleL2}; then we consider arbitrarily large values of $n_1,n_2$ in Sections~\ref{SecPreparation} through \ref{SecGeneric3}; this leaves some small exceptional cases which we deal with in Section~\ref{SecExceptions}. Finally, in Section~\ref{SecIsomorph} we show that all Lie algebra structures on these objects are isomorphic to the images of the listed ordinary Lie algebras. Note that Section~\ref{SecIsomorph} is the only part of the proof where we require $\bk$ to be algebraically closed; if this is not satisfied then there may be more isomorphism classes of Lie algebras on the objects listed in the theorem.

\subsection{Cases with invertible or odd components}\label{SecOddInv} If $n_1$ is even and $n_2$ is odd (or vice versa without loss of generality), then $\beta^{1,1}_2=\beta^{1,2}_0=\beta^{1,2}_1=0$ by the $p$-triangle condition and $\beta^{2,2}_0=0$ by antisymmetry. Equation~(\ref{EqnJacobi}) with $(i,j,k,l,r)=(2,2,1,0,n_1)$ gives $\beta^{1,1}_0\beta^{2,2}_1=0$, so either $\beta^{2,2}_1=0$ and $M_2$ is an ideal of $\fg$, or $\beta^{1,1}_0=0$ and $M_1\oplus M_2$ is an ideal, so $\fg$ is non-simple.

If $n_i\in\{0,p-2\}$ then $\beta^{0,i}_i=0$. Hence, if $n_1,n_2\in\{0,p-2\}$ then $M_1\oplus M_2$ is an ideal of $\fg$. Now suppose only one of $n_1,n_2$ is in $\{0,p-2\}$, and without loss of generality assume it is $n_2$. Then $\beta^{1,2}_2=0$ by the $p$-triangle condition, so if we can show that $\beta^{1,1}_2=0$ then $M_0\oplus M_1$ will be an ideal of $\fg$. If $n_2=p-2$ then this follows from the $p$-triangle condition; if instead $n_2=0$, then we can assume $n_1$ is even by the result above, so $\beta^{1,1}_2=0$ by antisymmetry.

Now suppose $n_1$ and $n_2$ are both odd and not equal to $p-2$. We have $\beta^{1,1}_0=\beta^{2,2}_0=0$ by antisymmetry, so if we can also show $\beta^{1,2}_0=0$ then $M_1\oplus M_2$ will be an ideal of $\fg$. Equation~(\ref{EqnJacobi}) with $(i,j,k,l)=(1,1,2,1)$ is
\[\left((-1)^{n_1-r/2}-1)\beta^{1,0}_1\beta^{1,2}_0\alpha^{n_1,n_1,r}_{n_2,n_1,2}\right)=0.\]
If $n=n_1=n_2$, then taking $r=0$ gives $\alpha^{n,n,0}_{n,n,2}=\frac{6n}{(n+1)(n+2)}$ which is non-zero for all $n<p-2$. If instead $n_1=n$ and $n_2=n+2$ for some integer $n$ (or vice versa without loss of generality), then taking $r=4$ gives $\alpha^{n,n,4}_{n+2,n,2}=-\frac{3(n+4)}{n(n+1)(n+2)}$, which is also non-zero since $n+2<p-2$. So in both cases, the equation above is non-trivial and gives $\beta^{1,2}_0=0$ as required. Finally, if $|n_1-n_2|>2$, then $\beta^{1,2}_0=0$ by the $p$-triangle condition.

\subsection{Cases with multiple $L_2$}\label{SecMultipleL2} From now on we can assume $n_1,n_2$ are even and non-zero. If $n_0=n_1=n_2=2$, then $\fg=A\otimes L_2$ for a 3-dimensional commutative algebra $A$ by Lemma~\ref{LemL2Alg}. This means $A$ is non-simple and has a non-zero proper ideal $I$, so $I\otimes L_2$ is an ideal of $\fg$.

Now suppose $n_1=2$ and $n_2=n\geq4$. Then $\beta^{1,1}_2=0$ (by antisymmetry for $n=4$ and by the triangle condition for $n>4$), and $M_0\oplus M_1$ forms a subalgebra. By Lemma~\ref{LemL2Alg} this subalgebra is $A\otimes L_2$ for a 2-dimensional commutative algebra $A$, so $A$ is isomorphic to either $\bk[x]/x^2$ or $\bk[x]/(x^2-x)$. By choosing a new decomposition of $\fg$ with $M_1=\bk x\otimes L_2$, we may assume $\beta^{1,1}_0=0$. We tackle the cases $n>4$ and $n=4$ separately.

If $n>4$, then $\beta^{1,2}_0=\beta^{1,2}_1=0$. Equation~(\ref{EqnJacobi}) with $(i,j,k,l,r)=(2,2,1,0,2)$ is
\[-2\beta^{2,2}_0\beta^{2,1}_2\alpha^{n,n,2}_{2,2,n}=-\frac{4}{n+2}\beta^{2,2}_0\beta^{2,1}_2=0.\]
If $\beta^{2,2}_0=0$ then $M_1\oplus M_2$ is an ideal, and if $\beta^{2,1}_2=0$ then $M_1$ is an ideal.

Now suppose $n=4$, meaning $\beta^{2,2}_2=0$. We compute Equation~(\ref{EqnJacobi}) for the following values of $(i,j,k,l,r)$:
\begin{align*}
(1,1,2,0,2):&\quad\beta^{1,2}_0\beta^{1,1}_1-2\beta^{1,2}_0\beta^{1,2}_2\alpha^{2,2,2}_{4,2,4}=\left(\beta^{1,1}_1-2\beta^{1,2}_2\right)\beta^{1,2}_0=0\\
(1,1,2,2,2):&\quad\beta^{1,2}_2\beta^{1,1}_1-2\beta^{1,2}_2\beta^{1,2}_2\alpha^{2,2,2}_{4,4,4}=\left(\beta^{1,1}_1-\frac23\beta^{1,2}_2\right)\beta^{1,2}_2=0\\
(2,2,1,0,2):&\quad-2\beta^{2,1}_0\beta^{2,1}_1\alpha^{4,4,2}_{2,2,2}-2\beta^{2,2}_0\beta^{2,1}_2\alpha^{4,4,2}_{2,2,4}=\frac9{10}\beta^{2,1}_0\beta^{2,1}_1-\frac23\beta^{2,2}_0\beta^{2,1}_2=0\\
(2,2,1,2,2):&\quad-2\beta^{2,0}_2\beta^{2,1}_0\alpha^{4,4,2}_{2,4,2}-2\beta^{2,1}_2\beta^{2,1}_1\alpha^{4,4,2}_{2,4,2}=-\frac9{10}\beta^{2,1}_0-\frac35\beta^{2,1}_2\beta^{2,1}_1=0
\end{align*}
If we multiply the first equation by $\beta^{1,2}_2$ and the second by $\beta^{1,2}_0$ and subtract, we obtain $\frac43(\beta^{1,2}_2)^2\beta^{1,2}_0=0$. If $\beta^{1,2}_2=0$ then the fourth equation gives $\beta^{1,2}_0=0$, so we have $\beta^{1,2}_0=0$ regardless. Then the third equation gives either $\beta^{2,2}_0=0$, in which case $M_1\oplus M_2$ is an ideal, or $\beta^{2,1}_2=0$, in which case $M_1$ is an ideal.

\subsection{Preparation for the generic cases}\label{SecPreparation} Now suppose $n_1,n_2$ are both even and greater than 2. First we show that $\beta^{1,2}_0=0$ in all cases. If $|n_1-n_2|\geq4$ then we have $\beta^{1,2}_0=0$ by the $p$-triangle condition, and if $n_1=n_2$ then we may assume $\beta^{1,2}_0=0$ by Lemma~\ref{LemGoodDecomp}. If instead $|n_1-n_2|=2$, by swapping $n_1$ and $n_2$ we can assume without loss of generality that $n_1/2$ is even, so $\beta^{1,1}_1=\beta^{2,2}_1=0$ by antisymmetry. Taking $(i,j,k,l,r)=(1,1,2,1,2)$ gives
\[\beta^{0,1}_1\beta^{1,2}_0\alpha^{n_1,n_1,2}_{n_2,n_1,2}+\beta^{1,2}_1\beta^{1,2}_2\alpha^{n_1,n_1,2}_{n_2,n_1,n_2}=0\]
after rescaling by $-\frac12$, while taking $(i,j,k,l,r)=(1,2,1,1,2)$ gives
\[\beta^{0,1}_1\beta^{1,2}_0\left(1+\alpha^{n_1,n_2,2}_{n_1,n_1,2}\right)+\beta^{1,2}_1\beta^{1,2}_2\alpha^{n_1,n_2,2}_{n_1,n_1,n_2}=0.\]
Using Equation~(\ref{EqnPerms}), $\alpha^{n_1,n_1,2}_{n_2,n_1,2}=\alpha^{n_1,n_2,2}_{n_1,n_1,2}$ and $\alpha^{n_1,n_1,2}_{n_2,n_1,n_2}=\alpha^{n_1,n_2,2}_{n_1,n_1,n_2}$, so we can subtract these equations to obtain $\beta^{1,2}_0=0$.

From now on we write $n_1=2x$ and $n_2=2y$ with $x,y\in\mN$ and $2\leq x\leq y$. The following special cases of Equation~(\ref{EqnJacobi}) will be useful. We will say that an odd integer $u\leq\frac{p-3}2$ is \textbf{admissible to $(x,y)$} if $u\neq x$, $u\neq y$ and $0\leq u\leq 2x$. If $u$ is admissible to $(x,y)$, then taking $(i,j,k,l,r)$ to be either $(1,1,2,2,2u)$ or $(2,2,1,1,2u)$ in Equation~(\ref{EqnJacobi}) gives
\begin{equation}\label{EqnJacobiCut}
\begin{split}
f_u(x,y,x)g(x,y,x)\beta^{1,1}_2\beta^{1,2}_1+f_u(x,y,y)g(x,y,y)(\beta^{1,2}_2)^2&=\delta_{u=1}\frac{y+1}4\beta^{1,1}_0,\\
f_u(x,y,x)g(x,y,x)(\beta^{1,2}_1)^2+f_u(x,y,y)g(x,y,y)\beta^{2,2}_1\beta^{2,1}_2&=\delta_{u=1}\frac{x+1}4\beta^{2,2}_0
\end{split}
\end{equation}
while taking $(i,j,k,l,r)$ to be $(1,1,1,1,2u)$ gives
\begin{equation}\label{Eqn1111}
\begin{split}\left(f_u(x,x,1)g(x,x,1)-\frac{\delta_{u=1}}2\right)\frac{x+1}2\beta^{1,1}_0&\\
+f_u(x,x,x)g(x,x,x)(\beta^{1,1}_1)^2&+f_u(x,x,y)g(x,x,y)\beta^{1,1}_2\beta^{1,2}_1=0.\end{split}
\end{equation}
using the functions $f,g$ defined in Section~\ref{SecRational}. Since $\beta^{1,2}_0=0$, we note that:
\begin{enumerate}
\item $M_1\oplus M_2$ is an ideal of $\fg$ if and only if $\beta^{1,1}_0=\beta^{2,2}_0=0$,
\item $M_1$ is an ideal of $\fg$ if and only if $\beta^{1,1}_0=\beta^{1,1}_2=\beta^{1,2}_2=0$, and
\item $M_2$ is an ideal of $\fg$ if and only if $\beta^{2,2}_0=\beta^{2,2}_1=\beta^{1,2}_1=0$.
\end{enumerate}
We will frequently use these to rule out non-simple algebras.

\subsection{The large $y$ generic case}\label{SecGeneric1} Suppose $y>2x$, so that $\beta^{1,2}_1=\beta^{1,1}_2=0$. This means Equation~(\ref{EqnJacobiCut}) reduces to
\[f_u(x,y,y)g(x,y,y)(\beta^{1,2}_2)^2=\delta_{u=1}\frac{y+1}4\beta^{1,1}_0,\quad f_u(x,y,y)g(x,y,y)\beta^{2,2}_1\beta^{2,1}_2=\delta_{u=1}\frac{x+1}4\beta^{2,2}_0.\]
Note that $x+2y\leq p-2$, since otherwise $g(x,y,y)=0$ and taking $u=1$ gives $\beta^{1,1}_0=\beta^{2,2}_0=0$. Now, if there is some $u\geq3$ for which $f_u(x,y,y)\neq0$, then $\beta^{1,2}_2=\beta^{2,2}_1\beta^{2,1}_2=0$ and then the above equation with $u=1$ gives $\beta^{1,1}_0=\beta^{2,2}_0=0$. So in order for $\fg$ to be simple, we require $f_u(x,y,y)=0$ for all $u$ admissible to $(x,y)$. We now determine when this holds:
\begin{enumerate}
\item If $x=2$, then $y\geq5$ and $p\geq13$. Then $f_3(2,y,y)=\frac{3(y+2)(2y+1)}{5y(2y-1)}\neq0\mod p$.
\item If $x=3$, then $y\geq7$ and $p\geq17$. Then $f_5(3,y,y)=-\frac{5(y+3)(2y+1)(2y+5)}{42y(y-1)(2y-1)}$, which is non-zero unless $2y=p-5$. This gives the object $L_2\oplus L_6\oplus L_{p-5}$ which is case (\ref{Item3D}) in the theorem.
\item If $x=5$, then $y\geq11$ and $p\geq29$. Then $f_3(5,y,y)=\frac{7(y-3)(y+4)(2y+1)}{10y(y-1)(2y-1)}\neq0\mod p$.
\item If $x=7$, then $y\geq15$ and $p\geq37$. Then
\[125(4y^2+4y-153)f_3(7,y,y)+1386(y-2)(2y-3)f_5(7,y,y)=\frac{52785(2y+1)}{y(y-1)(2y-1)}\]
is non-zero modulo $p$, since the highest prime factor of 52785 is 23. So at least one of $f_3,f_5$ is non-zero.
\item Finally, suppose that either $x\in\{4,6\}$ or $x\geq8$, so $y\geq9$ and $p\geq23$. Let $f'_u(x,y,y)$ be the largest factor of the numerator of $f_u(x,y,y)$ with no linear components. The following computational results were obtained using SageMath. All linear components appearing in $f_u(x,y,y)$ for $u\in\{3,5,7\}$ are of the form $2y+1$, $x+n$, $y-n$, $2x-n$ or $2y-n$ with $0\leq n\leq 5$, none of which can be zero modulo $p$. Moreover, the ideal in $\mZ[x,y]$ generated by $f'_u(x,y,y)$ for $u\in\{3,5,7\}$ has a Gr\"obner basis containing the polynomials
\[h_1=51480(4y+1)(4y+3),\quad h_2=5x^2+5x-12y^2-12y-6.\]
Thus, in order for all the $f$'s to be zero modulo $p$, $h_1$ and $h_2$ must be zero modulo $p$. The largest prime factor of 51480 is 13, so either $y=(p-1)/4$ or $y=(p-3)/4$, but substituting either of these into $h_2$ gives $\frac54(2x-1)(2x+3)$, giving solutions $x=(p+1)/2$ and $x=(p-3)/2$ which both contradict $x<y$. So this case has no solutions.
\end{enumerate}

\subsection{The $x=y$ generic case}\label{SecGeneric2} Next, suppose $x=y$. This means Equation~(\ref{EqnJacobiCut}) reduces to
\begin{align*}
f_u(x,x,x)g(x,x,x)\left(\beta^{1,1}_2\beta^{1,2}_1+(\beta^{1,2}_2)^2\right)&=\delta_{u=1}\frac{x+1}4\beta^{1,1}_0,\\
f_u(x,x,x)g(x,x,x)\left((\beta^{1,2}_1)^2+\beta^{2,2}_1\beta^{2,1}_2\right)&=\delta_{u=1}\frac{x+1}4\beta^{2,2}_0
\end{align*}
We assume $3x\leq p-2$, since otherwise $g(x,x,x)=0$ and $\beta^{1,1}_0=\beta^{2,2}_0=0$. If $f_u(x,x,x)\neq0$ for some $u\neq1$ then we obtain
\[\beta^{1,1}_2\beta^{1,2}_1+(\beta^{1,2}_2)^2=0\text{ and }(\beta^{1,2}_1)^2+\beta^{2,2}_1\beta^{2,1}_2=0,\]
and then by considering $u=1$ we get $\beta^{1,1}_0=\beta^{2,2}_0=0$. So in order for $\fg$ to be simple, we require $f_u(x,x,x)=0$ for all $u$ admissible to $(x,x)$. We now determine when this holds:
\begin{enumerate}
\item If $x=2$ then $f_3(2,2,2)=2$ which is non-zero modulo every $p>2$.
\item If $x=3$ then $f_5(3,3,3)=-11/6$ which is non-zero modulo every $p>11$. For $p=11$, $x=3$ gives the object $\fg=L_2\oplus L_6\oplus L_6$ which is case (\ref{Item3D}) in the theorem.
\item If $x=5$ then $f_3(5,5,5)=77/100$ which is not zero modulo every $p>2x+2=12$.
\item If $x=4$ or $x\geq6$, then we can compute
\begin{align*}
f_3(x,x,x)&=\frac{(x+1)(x+2)(2x+1)(7x^2+7x+6)}{80x(x-1)(2x-1)^2},\\
f_5(x,x,x)&=-\frac{(x+1)(x+2)(x+3)(2x+1)(23x^4+46x^3+707x^2+684x+360)}{4032x(x-1)(x-2)(2x-1)^2(2x-3)^2}.
\end{align*}
Note that the linear factors in these expressions are non-zero modulo $p$, and the non-linear factors $f'_3=7x^2+7x+6$ and $f'_5=23x^4+46x^3+707x^2+684x+360$ satisfy
\[(161x^2+161x+4650)f'_3-49f'_5=10260,\]
which has prime factors $2,3,5,19$. Thus, in order to have $f_3(x,x,x)=f_5(x,x,x)=0$ we must have $p=19$, but $f'_3\neq0$ modulo $19$ for all integer values of $x$, and thus we have no solutions.
\end{enumerate}

\subsection{The small $y$ generic case}\label{SecGeneric3} Finally, suppose $x<y\leq 2x$. We can rearrange Equation~(\ref{EqnJacobiCut}) into the matrix equations
\begin{equation}\label{EqnMatrix}
Fv_1=\begin{bmatrix}\beta^{1,1}_0(y+1)/4\\0\\0\\\vdots\end{bmatrix},\qquad
Fv_2=\begin{bmatrix}\beta^{2,2}_0(x+1)/4\\0\\0\\\vdots\end{bmatrix}
\end{equation}
where the matrix $F$ and vectors $v_1,v_2$ are
\[F=\begin{bmatrix}f_1(x,y,x) & f_1(x,y,y)\\f_3(x,y,x) & f_3(x,y,y)\\f_5(x,y,x) & f_5(x,y,y)\\\vdots & \vdots\end{bmatrix},\quad
v_1=\begin{bmatrix}g(x,y,x)\beta^{1,1}_2\beta^{1,2}_1\\g(x,y,y)(\beta^{1,2}_2)^2\end{bmatrix},\quad 
v_2=\begin{bmatrix}g(x,y,x)(\beta^{1,2}_1)^2\\g(x,y,y)\beta^{2,2}_1\beta^{2,1}_2\end{bmatrix}\]
with rows indexed by $u$ admissible to $(x,y)$. Write $d(u_1,u_2)$ for the determinant of the $2\times2$ submatrix of $F$ with rows corresponding to admissible $u_1,u_2$. If $d(u_1,u_2)$ is non-zero for some $u_1,u_2$ with $3\leq u_1<u_2$, then we can invert this submatrix in Equation~(\ref{EqnMatrix}) to obtain $v_1=v_2=0$, and then the first row gives $\beta^{1,1}_0=\beta^{2,2}_0=0$. So in order for $\fg$ to be simple, we require $d(u_1,u_2)=0$ for all $u_1,u_2$ admissible to $(x,y)$.

Let us assume $x\geq12$ so that every odd $u$ between 3 and 11 is admissible to $(x,y)$, and deal with $x\leq11$ in the next section. Let $d'(u_1,u_2)$ be the largest factor of the numerator of $d(u_1,u_2)$ with no linear components. It will suffice to consider the values $(u_1,u_2)\in U\coloneqq\{(3,5),(3,7),(7,9),(9,11)\}$. The following computational results were obtained using SageMath. All linear components appearing in $d(u_1,u_2)$ for $(u_1,u_2)\in U$ are of the form $2x+1$, $2y+1$, $x+y+1$, $y-x$, $x+n$ or $y+n$ with $0\leq n\leq 5$, and none of these can be zero modulo $p$. Moreover, the ideal generated by $d'(u_1,u_2)$ for $(u_1,u_2)\in U$ has a Gr\"obner basis containing the polynomials $c h_1(x,y)$ and $c h_2(x,y)$ where
\begin{align*}
c&= 2^{15}\times3^9\times5^4\times7^3\times11^2\times13^3\times17^2\times19,\\
h_1(x,y)&= 16y^2 + 16y + 16x^2 + 16x - 9,\\
h_2(x,y)&= 592y^4 + 1184y^3 + 255y^2 - 337y - 4x^2 - 4x - 81.
\end{align*}
Since $p>2x\geq24$, in order to have $d'(u_1,u_2)=0$ for all $(u_1,u_2)\in U$ we require $h_1(x,y)=h_2(x,y)=0$. Now, let
\[h_3(x,y)=h_1(x,y)+4h_2(x,y)=37(2y-1)(2y+3)(4y+1)(4y+3).\]
If $p=37$ then $h_1(x,y)=0$ is only satisfied by $(x,y)=(12,15)$, and for this case it can be computed that $d(3,7)=31\neq0$ modulo 37, so we may assume $p\neq37$. To have $h_3(x,y)=0$ we require $y\in\{\frac{p-3}2,\frac{p-3}4,\frac{p-1}4\}$. Substituting $y=\frac{p-3}2$ into $h_1(x,y)=0$ and solving for $x$ modulo $p$ gives $x\in\{\frac{p-1}4,\frac{p-3}4\}$, and these cases are addressed below. If we instead take $y=\frac{p-3}4$ or $\frac{p-1}4$ then we obtain $x\in\{\frac{p+1}2,\frac{p-3}2\}$, but this contradicts $x<y$.

We now show that $y=\frac{p-3}2$ and $x\in\{\frac{p-1}4,\frac{p-3}4\}$ does not admit a simple Lie algebra for $p\geq31$ (this includes some cases with $x\leq11$ so that we do not have to deal with them later). We have $x+2y\geq p-1$ and hence $\beta^{1,2}_2=\beta^{2,2}_1=\beta^{2,2}_2=0$. Using Equation~(\ref{EqnJacobiCut}) with $u=1$ and Equation~(\ref{Eqn1111}) with $u=3,5$, we obtain the matrix equation $F'v'=0$ where
\[F'=\begin{bmatrix}
-\frac{y+1}{2(x+1)g(x,x,1)}&0&f_1(x,y,x)\\
f_3(x,x,1) & f_3(x,x,x) & f_3(x,x,y)\\
f_5(x,x,1) & f_5(x,x,x) & f_5(x,x,y)
\end{bmatrix},\quad 
v'=\begin{bmatrix} \frac{x+1}2g(x,x,1)\beta^{1,1}_0 \\ g(x,x,x)(\beta^{1,1}_1)^2 \\ g(x,x,y)\beta^{1,1}_2\beta^{1,2}_1 \end{bmatrix}.\]
Substituting $(x,y)=(-\frac14,-\frac32)$ gives $\det F'=\frac{27115}{294912}$, while $(-\frac34,-\frac32)$ gives $\det F'=\frac{2465}{476872704}$. Both of these are well-defined and non-zero modulo $p$ for $p\geq31$, so we can invert this matrix to obtain $\beta^{1,1}_0=\beta^{1,1}_1=\beta^{1,1}_2\beta^{1,2}_1=0$. If $\beta^{1,1}_2=0$ then $M_1$ is an ideal, and if $\beta^{1,2}_1=0$ then the first row of Equation~(\ref{EqnMatrix}) gives $\beta^{2,2}_0=0$ and hence $M_1\oplus M_2$ is an ideal, so there are no solutions.

\subsection{Exceptional cases}\label{SecExceptions} Finally, we consider cases with $x\leq 11$ and $x<y\leq2x$. First suppose $x\leq 3$. The cases $(x,y)=(2,3),(3,5)$ appear in the theorem as cases (\ref{Item3A}), (\ref{Item3B}) and (\ref{Item3C}). For $(x,y)=(2,4),(3,4),(3,6)$, we have $\beta^{1,1}_2=0$ by antisymmetry. We have
\[\textstyle f_3(2,4,4)=\frac{81}{70}=\frac{3^4}{70},\quad f_5(3,4,4)=-\frac{65}{56}=-\frac{5\times13}{56},\quad f_5(3,6,6)=-\frac{221}{308}=-\frac{13\times17}{308},\]
so for sufficiently large $p$, Equation~(\ref{EqnJacobiCut}) gives $\beta^{1,2}_2=0$ for either $u=3$ or $u=5$. Then taking $u=1$ gives $\beta^{1,1}_0=0$ meaning $M_1$ is an ideal. The small $p$ exceptions are $(x,y,p)=(3,4,11),(3,4,13),(3,6,17)$; the first of these has $g(x,x,y)=0$ and thus $\beta^{1,1}_0=\beta^{2,2}_0=0$ by Equation~\ref{EqnJacobiCut} with $u=1$, and the other two are case (\ref{Item3D}) in the theorem.

Now consider $x\geq4$. The method at the start of Section~\ref{SecGeneric3} gives an algorithm for filtering out certain cases: given $(x,y,p)$, if we can find a pair of integers $u_1,u_2\leq\frac{p-3}2$ admissible to $(x,y)$ such that $d(u_1,u_2)\neq0$ modulo $p$, then there is no simple based Lie algebra structure on $L_2\oplus L_{2x}\oplus L_{2y}\in\Ver_p$. For some $p'\in\mN$, if one can find a list of integers $u_1,u_2,\dots\leq\frac{p'-3}2$ all admissible to $(x,y)$ such that the numerators of $d(u_1,u_2),d(u_3,u_4),\dots$ over $\mQ$ share no common prime factors $p\geq p'$, then this will eliminate $(x,y,p)$ for all $p\geq p'$. One can compute that $d(3,5)$, $d(3,7)$ and $d(3,9)$ are sufficient to eliminate the majority of cases using this argument, and then $d(3,11)$ covers the case $(x,y,p)=(7,9,53)$. The remaining exceptions (excluding cases appearing in the theorem) are $(x,y,p)$ as follows:
\[(4,5,13),(4,5,17),(4,8,19),(5,6,19),(5,7,17),(5,7,19),(5,8,23),(5,10,23),(6,7,23).\]
We shall examine these separately.
\begin{enumerate}
\item For $(x,y,p)=(4,5,13),(4,8,19),(5,7,17),(5,10,23)$, we have $\beta^{1,1}_2=\beta^{1,2}_2=0$ by either antisymmetry or the $p$-triangle condition in each case, so Equation~(\ref{EqnJacobiCut}) with $u=1$ gives $\beta^{1,1}_0=0$. Thus $M_1$ is an ideal and $\fg$ is non-simple.
\item For $(x,y,p)=(5,6,19),(5,8,23)$, we have $\beta^{1,1}_2=0$ by antisymmetry and $f_3(x,y,y)\neq0$ modulo $p$, so Equation~(\ref{EqnJacobiCut}) with $u=3$ gives $\beta^{1,2}_2=0$. Then taking $u=1$ gives $\beta^{1,1}_0=0$, so $M_1$ is an ideal and $\fg$ is non-simple.
\item For $(x,y,p)=(5,7,19)$, we have $\beta^{1,2}_2=0$ by the $p$-triangle condition and $f_3(x,y,x)\neq0$ modulo $p$. So Equation~(\ref{EqnJacobiCut}) with $u=3$ gives $\beta^{1,1}_2\beta^{1,2}_1=0$, and then taking $u=1$ gives $\beta^{1,1}_0=0$. If $\beta^{1,1}_2=0$ then $M_1$ is an ideal, and if instead $\beta^{1,2}_1=0$ then the second line of Equation~(\ref{EqnJacobiCut}) gives $\beta^{2,2}_0=0$ and $M_1\oplus M_2$ is an ideal, so $\fg$ is non-simple.
\item For $(x,y,p)=(4,5,17),(6,7,23)$, we have $\beta^{1,1}_1=0$ by antisymmetry. Equation~(\ref{EqnJacobiCut}) with $u=1,3$ and Equation~(\ref{Eqn1111}) with $u=3$ gives the equation $F''v''=0$ where
\[F''=\begin{bmatrix}
-\frac{y+1}{2(x+1)g(x,x,1)} & f_1(x,y,x) & f_1(x,y,y) \\
0 & f_3(x,y,x) & f_3(x,y,y) \\
f_3(x,x,1) & f_3(x,x,y) & 0
\end{bmatrix},\quad
v''=\begin{bmatrix} \frac{x+1}2g(x,x,1)\beta^{1,1}_0 \\ g(x,x,y)\beta^{1,1}_2\beta^{1,2}_1 \\ g(x,y,y)(\beta^{1,2}_2)^2 \end{bmatrix}.\]
One can compute that $F''$ has non-zero determinant in each case, so $\beta^{1,1}_0=\beta^{1,1}_2\beta^{1,2}_1=\beta^{1,2}_2=0$. If $\beta^{1,1}_2=0$ then $M_1$ is an ideal, and if instead $\beta^{1,2}_1=0$ then the second line of Equation~(\ref{EqnJacobiCut}) gives $\beta^{2,2}_0=0$ and $M_1\oplus M_2$ is an ideal, so $\fg$ is non-simple.
\end{enumerate}

\subsection{Lie algebra isomorphisms}\label{SecIsomorph} For each of the objects $M_0\oplus M_1\oplus M_2$ listed in the theorem, we now need to determine how many isomorphism classes of Lie algebras there are. Rescaling $M_1$ by $\lambda_1$ and $M_2$ by $\lambda_2$ has the effect of replacing $\beta^{i,j}_k$ with $\frac{\lambda_i\lambda_j}{\lambda_k}\beta^{i,j}_k$ (taking $\lambda_0=1$). This will allow us to normalise certain scalars $\beta^{i,j}_k$, and we show that in each case this will uniquely determine the remaining scalars (a more complicated strategy will be required for $L_2\oplus L_6\oplus L_6$, which corresponds to $p=11$ in case (\ref{Item3D})). Note that in all cases $\beta^{1,2}_0=0$ by the reasoning in Section~\ref{SecPreparation}.

\begin{enumerate}
\item For $L_2\oplus L_4\oplus L_6$, we have $\beta^{1,1}_1=\beta^{2,2}_1=0$ by antisymmetry. We get the following equations from the following values of $(i,j,k,l,r)$ in Equation~(\ref{EqnJacobi}):
\begin{align*}
(1,1,1,2,2):\quad&\textstyle 3\beta^{1,1}_2\beta^{1,2}_2=0,\\
(1,1,2,2,6):\quad&\textstyle \beta^{2,2}_2\beta^{1,1}_2+\frac37\beta^{1,1}_2\beta^{1,2}_1-(\beta^{1,2}_2)^2=0,\\
(1,2,1,0,2):\quad&\textstyle \frac43\beta^{1,1}_0\beta^{1,2}_1-\beta^{2,2}_0\beta^{1,1}_2=0,\\
(2,2,1,1,2):\quad&\textstyle \frac32\beta^{2,2}_0-\frac{10}7(\beta^{1,2}_1)^2=0,\\
(2,2,2,2,2):\quad&\textstyle \frac37\beta^{2,2}_0-\frac{20}9(\beta^{2,2}_2)^2=0.
\end{align*}
If $\beta^{2,2}_2=0$ then we get $\beta^{2,2}_0=\beta^{1,2}_1=0$ so that $M_2$ is an ideal. If $\beta^{1,1}_0=0$ then we get either $\beta^{2,2}_0=0$ so that $M_1\oplus M_2$ is an ideal, or $\beta^{1,1}_2=\beta^{1,2}_2=0$ so that $M_1$ is an ideal. So $\beta^{1,1}_0$ and $\beta^{2,2}_2$ are non-zero, and we can normalise to get $\beta^{1,1}_0=\beta^{2,2}_2=1$. Then the equations imply $\beta^{2,2}_0=\frac{140}{27}$, $\beta^{1,2}_2=0$, $\beta^{1,1}_2=-\frac35$ and $\beta^{1,2}_1=-\frac73$, so there is only one simple Lie algebra up to isomorphism.

\item For $L_2\oplus L_6\oplus L_{10}$ with $p=13$, we have $\beta^{1,2}_2=\beta^{2,2}_1=\beta^{2,2}_2=0$ by the $p$-triangle condition. We have the following equations:
\begin{align*}
(1,1,1,1,2):\quad&\textstyle 6\beta^{1,1}_0+5(\beta^{1,1}_1)^2+\beta^{1,1}_2\beta^{1,2}_1=0,\\
(1,1,2,2,2):\quad&\textstyle 3\beta^{1,1}_0+4\beta^{1,1}_2\beta^{1,2}_1=0,\\
(1,2,2,1,4):\quad&\textstyle 5\beta^{2,2}_0+10(\beta^{1,2}_1)^2=0.
\end{align*}
If $\beta^{1,2}_1=0$ then $\beta^{1,1}_0=\beta^{2,2}_0=0$ and $M_1\oplus M_2$ is an ideal, so $\beta^{1,2}_1\neq0$. If $\beta^{1,1}_1=0$ then the first two equations give $\beta^{1,1}_0=\beta^{1,1}_2=0$ and $M_1$ is an ideal, so $\beta^{1,1}_1\neq0$. This means we can first rescale $M_1$ to normalise $\beta^{1,1}_1=1$, and then rescale $M_2$ to normalise $\beta^{1,2}_1=1$. Then the equations imply $\beta^{1,1}_0=4$, $\beta^{1,1}_2=10$ and $\beta^{2,2}_0=11$, so there is only one simple Lie algebra up to isomorphism.

\item For $L_2\oplus L_6\oplus L_{10}$ with $p\geq17$, we have the following equations:
\begin{align*}
\text{(a)}&&(1,1,1,1,2):\quad&\textstyle \frac37\beta^{1,1}_0-\frac{20}9(\beta^{1,1}_1)^2+\beta^{1,1}_2\beta^{1,2}_1=0,\\
\text{(b)}&&(1,1,2,2,2):\quad&\textstyle 3\beta^{1,1}_0-\frac{21}{11}\beta^{1,1}_2\beta^{1,2}_1-\frac{240}{91}(\beta^{1,2}_2)^2=0,\\
\text{(c)}&&(1,1,2,2,6):\quad&\textstyle \beta^{1,1}_1\beta^{1,2}_2+\frac{102}{91}(\beta^{1,2}_2)^2=0,\\
\text{(d)}&&(1,1,2,2,10):\quad&\textstyle \beta^{1,1}_2\beta^{2,2}_2+\frac{13}{33}\beta^{1,1}_2\beta^{1,2}_1-\frac{2000}{1911}(\beta^{1,2}_2)^2=0,\\
\text{(e)}&&(2,2,1,0,6):\quad&\textstyle \beta^{1,1}_0\beta^{2,2}_1-\frac23\beta^{1,2}_2\beta^{2,2}_0=0,\\
\text{(f)}&&(2,2,1,2,10):\quad&\textstyle \frac{25}{13}\beta^{1,2}_2\beta^{2,2}_2-\frac{10}{33}\beta^{1,2}_1\beta^{1,2}_2=0,\\
\text{(g)}&&(2,2,2,2,2):\quad&\textstyle \frac{78}{55}\beta^{2,2}_0-\frac{576}{143}\beta^{1,2}_2\beta^{2,2}_1-\frac{567}{65}(\beta^{2,2}_2)^2=0.
\end{align*}
If $\beta^{2,2}_2=0$, then either $\beta^{1,2}_1=0$ or $\beta^{1,2}_2=0$ by (f), but the first of these implies the second by (d). We then get $\beta^{1,1}_2\beta^{1,2}_1=0$ by (d) which implies $\beta^{1,1}_0=0$ by (b) and $\beta^{2,2}_0=0$ by (g), so $M_1\oplus M_2$ is an ideal. Thus $\beta^{2,2}_2\neq0$, so we can rescale $M_2$ to normalise $\beta^{2,2}_2=1$. First suppose $\beta^{1,2}_2\neq0$; then we can rescale $M_1$ to normalise $\beta^{1,2}_2=1$, and then the equations above imply $\beta^{1,2}_1=\frac{165}{26}$, $\beta^{1,1}_2=\frac{4000}{13377}$, $\beta^{1,1}_0=\frac{17280}{8281}$, $\beta^{1,1}_1=-\frac{102}{91}$, $\beta^{2,2}_0=\frac{56133}{845}$, $\beta^{2,2}_1=\frac{33957}{1600}$ giving a unique solution. Now suppose $\beta^{1,2}_2=0$. If $\beta^{1,1}_1=0$, then equations (a) and (b) give $\beta^{1,1}_0=\beta^{1,1}_2\beta^{1,2}_1=0$, so that equation (d) gives $\beta^{1,1}_2=0$ and thus $M_1$ is an ideal. So $\beta^{1,1}_1\neq0$ and we can rescale $M_1$ to normalise $\beta^{1,1}_1=1$, and then the equations above give the unique solution $\beta^{1,1}_0=\frac{10}9$, $\beta^{2,2}_1=0$, $\beta^{1,1}_2=-\frac{130}{189}$, $\beta^{1,2}_1=-\frac{33}{13}$, $\beta^{2,2}_0=\frac{2079}{338}$. Hence we have shown there are exactly 2 simple Lie algebra structures on $L_2\oplus L_6\oplus L_{10}$ up to isomorphism. By computing the decompositions of $\fso_7$ and $\fsp_6$ over principal $\fsl_2$ subalgebras, one can check that the case with $\beta^{2,2}_1=\beta^{1,2}_2=0$ corresponds to $\fso_7$ while the other case corresponds to $\fsp_6$.

\item For $\fg=L_2\oplus L_6\oplus L_6$ with $p=11$, we have the following equations:
\begin{align*}
\text{(a)}&&(1,1,1,1,2):\quad&\textstyle 2\beta^{1,1}_0-\beta^{1,1}_2\beta^{1,2}_1-(\beta^{1,1}_1)^2=0,\\
\text{(b)}&&(1,1,2,2,2):\quad&\textstyle 2\beta^{1,1}_0-\beta^{1,1}_2\beta^{1,2}_1-(\beta^{1,2}_2)^2=0,\\
\text{(c)}&&(1,1,2,1,6):\quad&\textstyle 2\beta^{1,1}_1\beta^{1,2}_1+\beta^{1,1}_2\beta^{2,2}_1+\beta^{1,2}_1\beta^{1,2}_2=0,\\
\text{(d)}&&(1,1,2,2,6):\quad&\textstyle \beta^{1,1}_2\beta^{1,2}_1+\beta^{1,1}_1\beta^{1,2}_2+(\beta^{1,2}_2)^2+\beta^{1,1}_2\beta^{2,2}_2=0,\\
\text{(e)}&&(1,2,2,2,4):\quad&\textstyle \beta^{1,1}_2\beta^{2,2}_1-\beta^{1,2}_1\beta^{1,2}_2=0,\\
\text{(f)}&&(2,2,1,1,2):\quad&\textstyle 2\beta^{2,2}_0-\beta^{1,2}_2\beta^{2,2}_1-(\beta^{1,2}_1)^2=0,\\
\text{(g)}&&(2,2,2,2,2):\quad&\textstyle 2\beta^{2,2}_0-\beta^{1,2}_2\beta^{2,2}_1-(\beta^{2,2}_2)^2=0.
\end{align*}
If $\beta^{1,1}_1=\beta^{2,2}_2=0$, then (a) and (b) give $\beta^{1,2}_2=0$ while (f) and (g) give $\beta^{1,2}_1=0$, and hence equations (a) and (g) give $\beta^{1,1}_0=\beta^{2,2}_0=0$ and $M_1\oplus M_2$ is an ideal. So at least one of $\beta^{1,1}_1$ or $\beta^{2,2}_2$ is non-zero. We assume that this is $\beta^{1,1}_1$ by swapping $M_1$ and $M_2$ if needed, and rescale $M_1$ to normalise $\beta^{1,1}_1=1$.

Next we show that it is possible to choose a new decomposition of $\fg$ in which $\beta^{2,2}_2=0$. If $\beta^{2,2}_2\neq0$, then normalise it to 1 by rescaling $M_2$. Subtracting equations (c) and (e) gives $\beta^{1,2}_1(1+\beta^{1,2}_2)=0$, and equations (f) and (g) imply that $(\beta^{1,2}_1)^2=1\neq0$, hence we have $\beta^{1,2}_2=-1$. Write $t\coloneqq\beta^{1,1}_2$, which is non-zero by equation (e). Then (d) gives $\beta^{1,2}_1=-1$, and the other equations give $\beta^{2,2}_1=\frac1t$, $\beta^{1,1}_0=\frac{1-t}2$ and $\beta^{2,2}_0=\frac{t-1}{2t}$. Note that $t\neq1$ or else $\beta^{1,1}_0=\beta^{2,2}_0=0$. We will now define submodules $N_1,N_2$ of $L_6\oplus L_6$ so that $M_0\oplus N_1\oplus N_2$ is another decomposition of $\fg$. The map $N_1\oplus N_2\to M_1\oplus M_2$ is described by a $2\times2$ matrix $\sM{a&b\\c&d}$, and we take $b=c=1$ and $d=at$. This ensures that $ab\beta^{1,1}_0+cd\beta^{2,2}_0=0$, meaning the bracket map $N_1\otimes N_2\to M_0$ is zero. Meanwhile, the bracket map $N_2\otimes N_2\to N_2$ corresponds to the bottom right entry of the matrix
\[\begin{pmatrix}a&b\\c&d\end{pmatrix}^{-1}\begin{pmatrix}\beta^{1,1}_1&\beta^{1,2}_1&\beta^{1,2}_1&\beta^{2,2}_1\\\beta^{1,1}_2&\beta^{1,2}_2&\beta^{1,2}_2&\beta^{2,2}_2\end{pmatrix}\begin{pmatrix}a^2&ab&ab&b^2\\ac&ad&bc&bd\\ac&bc&ad&bd\\c^2&cd&cd&d^2\end{pmatrix}\]
which evaluates to $a^3t^2-3a^2t+3at-1$ divided by $\det\sM{a&b\\c&d}=a^2t-1$. Let us take $a$ to be a root of this cubic polynomial. Note that if $t=\frac1{a^2}$ then this polynomial becomes $4(\frac1a-1)$, and hence $a=t=1$ contradicting $t\neq1$, thus $\det\sM{a&b\\c&d}\neq0$ and $N_1,N_2$ are well-defined. We have now found a decomposition of $\fg$ for which $\beta^{1,2}_0=\beta^{2,2}_2=0$, and by rescaling $N_1$ we can normalise $\beta^{1,1}_1$ to 1 again. Now equations (f) and (g) give $\beta^{1,2}_1=0$, and then (e) becomes $\beta^{1,1}_2\beta^{2,2}_1=0$. If $\beta^{2,2}_1=0$ then (f) gives $\beta^{2,2}_0=0$ so that $M_2$ is an ideal, so instead we must have $\beta^{1,1}_2=0$. Then (a) gives $\beta^{1,1}_0=\frac12$, so (b) and (d) give $\beta^{1,2}_2=-1$. Finally, we can rescale $N_2$ to normalise $\beta^{2,2}_1=1$ (once again if $\beta^{2,2}_1=0$ then $M_2$ is an ideal), and then $\beta^{2,2}_0=-\frac12$. Thus there is a unique simple Lie algebra up to isomorphism.

\item For $L_2\oplus L_6\oplus L_{p-5}$ with $p\geq13$, note that $\beta^{2,2}_2=0$ by the $p$-triangle condition. Similarly, $\beta^{1,1}_2=\beta^{1,2}_1=0$ for $p>17$. If instead $p\in\{13,17\}$, then antisymmetry shows that $\beta^{1,1}_2=0$, and we use the equation $(1,2,1,0,6)$ below for $\beta^{1,2}_1$. We have the following equations:
\begin{align*}
(1,1,1,1,2):\quad&\textstyle \frac37\beta^{1,1}_0-\frac{20}9(\beta^{1,1}_1)^2=0,\\
(1,1,2,2,2):\quad&\textstyle \frac{p-3}4\beta^{1,1}_0-\frac{80(p-6)(p-7)}{(p-1)(p-2)(p-3)}(\beta^{1,2}_2)^2=0,\\
(1,1,2,2,6):\quad&\textstyle \beta^{1,2}_2\beta^{1,1}_1+\frac{24(p^2-8p-3)}{(p-1)(p-2)(p-3)}(\beta^{1,2}_2)^2=0,\\
(1,2,1,0,6):\quad&\textstyle \beta^{1,1}_0\beta^{1,2}_1=0\text{ for }p\in\{13,17\},\\
(2,2,2,2,2):\quad&\textstyle \frac{(p-2)(p-7)(p-9)}{4(p-4)(p-5)}\beta^{2,2}_0-\frac{140(p+1)(p-6)(p-7)(p-9)}{(p-1)(p-2)(p-3)(p-4)(p-5)}\beta^{1,2}_2\beta^{2,2}_1=0.
\end{align*}
If $\beta^{1,1}_1=0$ then $\beta^{1,1}_0=\beta^{2,2}_0=0$ and $M_1\oplus M_2$ is an ideal, so $\beta^{1,1}_1\neq0$ and we can rescale $M_1$ to normalise $\beta^{1,1}_1=1$. The first two equations give $\beta^{1,1}_0\neq0$ (hence $\beta^{1,2}_1=0$) and $\beta^{1,2}_2\neq0$. If $\beta^{2,2}_0=0$ then $\beta^{2,2}_1=0$ and $M_2$ is an ideal, so $\beta^{2,2}_0\neq0$ and we can rescale $M_2$ to normalise $\beta^{2,2}_0=1$. Then the equations imply $\beta^{1,1}_0=\frac{140}{27}$, $\beta^{1,2}_2=-\frac1{12}$ and $\beta^{2,2}_1=\frac3{70}$. So there is only one simple Lie algebra up to isomorphism.

\item For $L_2\oplus L_8\oplus L_{14}$ with $p=17$, we have $\beta^{1,1}_1=0$ by antisymmetry and $\beta^{1,2}_2=\beta^{2,2}_1=\beta^{2,2}_2=0$ by the $p$-triangle condition. We have the following equations:
\begin{align*}
(1,1,1,1,2):\quad&\textstyle 8\beta^{1,1}_0+2\beta^{1,1}_2\beta^{1,2}_1=0,\\
(2,2,1,1,2):\quad&\textstyle 11\beta^{2,2}_0+(\beta^{1,2}_1)^2=0.
\end{align*}
If $\beta^{1,2}_1=0$ then $\beta^{1,1}_0=\beta^{2,2}_0=0$ and $M_1\oplus M_2$ is an ideal, so $\beta^{1,2}_1\neq0$ and we can rescale $M_2$ to normalise $\beta^{1,2}_1=1$. If $\beta^{1,1}_2=0$ then $\beta^{1,1}_0=0$ and $M_1$ is an ideal, so $\beta^{1,1}_2\neq0$ and we can rescale $M_1$ to normalise $\beta^{1,1}_2=1$. Then the equations imply $\beta^{1,1}_0=4$ and $\beta^{2,2}_0=3$, so there is only one simple Lie algebra up to isomorphism.
 
\item For $L_2\oplus L_{10}\oplus L_{14}$ with $p=23$, we have the following equations:
\begin{align*}
\text{(a)}&&(1,1,1,1,2):\quad&\textstyle \beta^{1,1}_0+4\beta^{1,1}_2\beta^{1,2}_1+21(\beta^{1,1}_1)^2=0,\\
\text{(b)}&&(1,1,2,2,2):\quad&\textstyle 4\beta^{1,1}_0+14\beta^{1,1}_2\beta^{1,2}_1+15(\beta^{1,2}_2)^2=0,\\
\text{(c)}&&(1,1,2,1,6):\quad&\textstyle 18\beta^{1,1}_1\beta^{1,2}_1+12\beta^{1,2}_1\beta^{1,2}_2=0,\\
\text{(d)}&&(1,1,2,2,6):\quad&\textstyle 3\beta^{1,1}_2\beta^{1,2}_1+22(\beta^{1,2}_2)^2=0,\\
\text{(e)}&&(2,2,1,1,6):\quad&\textstyle 22\beta^{1,2}_2\beta^{2,2}_1+3(\beta^{1,2}_1)^2=0,\\
\text{(f)}&&(2,2,1,1,2):\quad&\textstyle 3\beta^{2,2}_0+15\beta^{1,2}_2\beta^{2,2}_1+14(\beta^{1,2}_1)^2=0,\\
\text{(g)}&&(2,2,2,1,10):\quad&\textstyle 7\beta^{1,2}_1\beta^{2,2}_1+2\beta^{2,2}_1\beta^{2,2}_2=0.
\end{align*}
If $\beta^{1,1}_1=0$, then either $\beta^{1,2}_1=0$ or $\beta^{1,2}_2=0$ by (c), but these imply each other by (d) and (e). We then get $\beta^{1,1}_0=\beta^{2,2}_0=0$ by (a) and (f), and $M_1\oplus M_2$ is an ideal. Thus $\beta^{1,1}_1\neq0$ and we can rescale $M_1$ to normalise $\beta^{1,1}_1=1$. If $\beta^{2,2}_2=0$, then either $\beta^{1,2}_1=0$ or $\beta^{2,2}_1=0$ by (g), but the second implies the first by (e), and we get the same contradiction again. Thus $\beta^{2,2}_2\neq0$ and we can rescale $M_2$ to normalise $\beta^{2,2}_2=1$. Then the equations above imply $\beta^{1,1}_0=22$, $\beta^{1,1}_2=6$, $\beta^{1,2}_1=3$, $\beta^{1,2}_2=10$, $\beta^{2,2}_0=7$ and $\beta^{2,2}_1=5$, and there is only one simple Lie algebra up to isomorphism.

\item For $L_2\oplus L_{14}\oplus L_{26}$ with $p=29$, we have $\beta^{1,2}_2=\beta^{2,2}_1=\beta^{2,2}_2=0$ by the $p$-triangle condition. We have the following equations:
\begin{align*}
(1,1,1,1,2):\quad&\textstyle 19\beta^{1,1}_0+5\beta^{1,1}_2\beta^{1,2}_1+28(\beta^{1,1}_1)^2=0,\\
(1,1,1,1,6):\quad&\textstyle 19\beta^{1,1}_0+14\beta^{1,1}_2\beta^{1,2}_1+21(\beta^{1,1}_1)^2=0,\\
(2,2,1,1,2):\quad&\textstyle 4\beta^{2,2}_0+9(\beta^{1,2}_1)^2=0.
\end{align*}
If $\beta^{1,2}_1=0$ then $\beta^{1,1}_1=\beta^{1,1}_0=\beta^{2,2}_0=0$ and $M_1\oplus M_2$ is an ideal. Thus $\beta^{1,2}_1\neq0$, and we can rescale $M_2$ to normalise $\beta^{1,2}_1=1$. If $\beta^{1,1}_1=0$ then $\beta^{1,1}_0=\beta^{1,1}_2=0$ and $M_1$ is an ideal. Thus $\beta^{1,1}_1\neq0$ and we can rescale $M_1$ to normalise $\beta^{1,1}_1=1$. Then the equations above imply $\beta^{1,1}_0=28$, $\beta^{1,1}_2=4$ and $\beta^{2,2}_0=5$, so there is only one simple Lie algebra up to isomorphism.
\end{enumerate}

This completes the proof of Theorem~\ref{ThmLength3}.

\section{Computational results and exceptional cases}\label{SecComp}

In this section we describe a method to search for Lie algebras computationally, and give some results for length 4 Lie algebras in $\Ver_p$.

Let $\bk$ be algebraically closed with $\charr\bk=p\geq5$, and let $\fg=\bigoplus_i M_i$ and $\beta^{i,j}_k$ be as in Section~\ref{SecLieForm}. Fix the objects $M_i$, but let the scalars $\beta^{i,j}_k$ vary. The antisymmetry and Jacobi relations are polynomial equations in the variables $\beta^{i,j}_k$, and they generate an ideal $I_{\Lie}$ in $\bk[\beta^{i,j}_k]$. A Lie algebra structure on $\fg$ corresponds to a point in the affine algebraic set $V_{\Lie}=V(I_{\Lie})$. If the objects $M_i$ are pairwise non-isomorphic, then subobjects of $\fg$ are of the form $\bigoplus_{i\in S}M_i$ for a proper non-zero subset $S$ of the indices. Such a subobject is an ideal if and only if $\beta^{i,j}_k=0$ for all $i\in S$, $k\not\in S$ and $j$ arbitrary. If we define $I_S$ to be the ideal generated by all such $\beta^{i,j}_k$, and then define $I_{\rm n.s.}=\bigcap_S I_S$, then the algebraic set $V_{\rm n.s.}$ corresponding to $I_{\rm n.s.}$ is the set of all ``non-simple operations'' on $\fg$. This means there exists a simple Lie algebra structure on $\fg$ if and only if $V_{\Lie}\not\subseteq V_{\rm n.s.}$ if and only if $I_{\rm n.s.}\not\subseteq\sqrt{I_{\Lie}}$. Note that since the generators of $I_{\rm n.s.}$ have coefficients in $\mF_p$, it suffices to use the radical of $I_{\Lie}$ over $\mF_p$. If there are summands with multiplicity, more variables can be introduced to capture all possible subobjects. For example, if $\fg=M_1\oplus M_2$ with $M_1\cong M_2$, then any non-zero proper subobject not equal to $M_1$ is a linear combination of $M_1$ and $M_2$ with coefficients $\alpha$ and 1 for some $\alpha\in\bk$, and this is an ideal if and only if
\[\det\begin{pmatrix}\alpha\beta^{1,1}_1+\beta^{2,1}_1&\alpha\\\alpha\beta^{1,1}_2+\beta^{2,1}_2&1\end{pmatrix}=\det\begin{pmatrix}\alpha\beta^{1,2}_1+\beta^{2,2}_1&\alpha\\\alpha\beta^{1,2}_2+\beta^{2,2}_2&1\end{pmatrix}=0.\]
Thus the collection of non-simple algebras is now described by a variety over $\bk[\beta^{i,j}_k,\alpha]$.

Computing the ideals $I_{\rm n.s.}$ and $\sqrt{I_{\Lie}}$ and checking this inclusion is possible with computer algebra software, and below we report the results of an implementation of this method in SageMath. The following is a list of length 4 objects in $\Ver_p$ (and not in $\sVec\subset\Ver_p$) for $5\leq p\leq67$ which admit a simple based Lie algebra structure. Known Lie algebras on these objects are also listed, with non-isomorphic algebras on the same object listed separately.
\begin{itemize}
\newcounter{length4counter}
\item Principal restrictions of simple groups:
\begin{enumerate}
\item $L_2\oplus L_4\oplus L_6\oplus L_8=\fsl(L_4)$ for $p\geq11$ ($A_4$)
\item $L_2\oplus L_6\oplus L_6\oplus L_{10}=\fso(L_0\oplus L_6)$ for $p\geq13$ ($D_4$)
\item $L_2\oplus L_6\oplus L_{10}\oplus L_{14}=\fso(L_8)$ for $p\geq17$ ($B_4$)
\item $L_2\oplus L_6\oplus L_{10}\oplus L_{14}=\fso(L_{p-9})$ for $p\geq19$ ($C_4$)
\item $L_2\oplus L_6\oplus L_{10}\oplus L_{p-7}=\fso(L_0\oplus L_{p-7})$ for $p\geq19$ ($D_{(p-5)/2}$)
\item $L_2\oplus L_{10}\oplus L_{14}\oplus L_{22}$ for $p\geq29$ ($F_4$)
\item $L_2\oplus L_6\oplus L_{10}\oplus L_{10}=\fso(L_0\oplus L_{10})$ for $p=17$ ($D_6$)
\item $L_2\oplus L_8\oplus L_{10}\oplus L_{16}$ for $p=19$ ($E_6$)
\item $L_2\oplus L_{14}\oplus L_{26}\oplus L_{38}$ for $p=41$ ($E_8$)
\item $L_2\oplus L_{14}\oplus L_{22}\oplus L_{34}$ for $p=43$ ($E_8$)
\setcounter{length4counter}{\value{enumi}}
\end{enumerate}
\item Subregular restrictions of simple groups (the $a_1$ notation for the exceptional types follows \cite{St}):
\begin{enumerate}
\setcounter{enumi}{\value{length4counter}}
\item $L_0\oplus L_1\oplus L_1\oplus L_2=\fsl(L_0\oplus L_1)$ for $p\geq5$ ($A_2$ with a root subgroup)
\item $L_0\oplus L_2\oplus L_2\oplus L_2=\fso(L_{p-3}^{\oplus 2})$ for $p\geq5$ ($B_2$ with a short root subgroup)
\item\label{Item4B} $L_0\oplus L_2\oplus L_{p-3}\oplus L_{p-3}=\fso(L_0\oplus L_0\oplus L_{p-3})$ for $p\geq7$ ($B_{(p-1)/2}$ with a subregular nilpotent, or $\fosp(2|2)$ with $\phi$ given by $\fsl_2\cong\fsl(0|2)\hookrightarrow{\fosp(2|2)}$)
\item\label{Item4D} $L_2\oplus L_2\oplus L_{p-5}\oplus L_{p-3}=\fso(L_2\oplus L_{p-3})$ for $p\geq11$ ($D_{(p+1)/2}$ with a subregular nilpotent, $\fosp(3|2)$ with $\phi$ the diagonal map into the even part $\fsl_2\oplus\fsl_2$)
\item $L_2\oplus L_2\oplus L_2\oplus L_4$ for $p\geq7$ ($G_2(a_1)$)
\item $L_2\oplus L_4\oplus L_6\oplus L_{10}$ for $p=13$ ($E_6(a_1)$)
\item $L_2\oplus L_6\oplus L_{10}\oplus L_{16}$ for $p=19$ ($E_7(a_1)$)
\item $L_2\oplus L_{10}\oplus L_{18}\oplus L_{28}$ for $p=31$ ($E_8(a_1)$)
\setcounter{length4counter}{\value{enumi}}
\end{enumerate}
\item Restrictions of Lie superalgebras:
\begin{enumerate}
\setcounter{enumi}{\value{length4counter}}
\item $L_2\oplus L_4\oplus L_{p-6}\oplus L_{p-4}$ for $p\geq7$ ($Q(2)$)
\item $L_1\oplus L_2\oplus L_3\oplus L_4$ for $p=7$ ($H(5)$)
\end{enumerate}
\end{itemize}
All of these based Lie algebras $\fg$ lift to a $\Ver_p$-group $(G,\phi)$ as follows. If $\fg$ is invariantless, then we can apply Corollary~\ref{CorBased}. For the non-invariantless cases, each comes from an algebraic group $G$ and has one copy of the summand $\unit=L_0$ in $\Lie(G)$. The inclusion of this summand integrates to a coweight $\mG_m\to G$, giving an action of $\mG_m$ on $\Lie(G)$ and an isomorphism of $\Lie(\mG_m)$ with the summand $\unit$. The image in $\Ver_p$ of this action and isomorphism gives a Harish-Chandra pair structure on $(\mG_m,\fg)$. To obtain a $\Ver_p$-group we require a map $\phi_0:\mZ/2\to\mG_m$ inducing the parity action on $\fg$. For the case $L_0\oplus L_1\oplus L_1\oplus L_2$, write $\fg$ as the image of $\fsl_3$ restricted to the upper left Levi $\fsl_2$, take $\mG_m$ to be the coweight $(1,1,-2)$, and take $\phi_0$ to be the inclusion $\mZ/2\cong\mu_2\hookrightarrow\mG_m$. The other two cases are in $\Ver_p^+$, so we can take $\phi_0$ to be trivial.

\end{document}